\documentclass[12pt]{article}

\usepackage[top=0.85in,footskip=0.75in]{geometry}

\usepackage{amsthm,amsmath,amsfonts,amssymb}
\usepackage{natbib}
\usepackage[utf8]{inputenc}
\usepackage{url}
\usepackage{graphicx}
\usepackage{mathtools}
\usepackage{bbm}
\usepackage{bm}
\usepackage{multirow}
\usepackage{caption}
\usepackage{subcaption}
\usepackage{hyperref}

\theoremstyle{plain} 
\newtheorem{theorem}{Theorem}
\newtheorem{lemma}{Lemma}
\newtheorem{proposition}{Proposition}
\newtheorem{corollary}{Corollary}
\theoremstyle{remark} 
\newtheorem{definition}{Definition}
\newtheorem*{remark}{Remark}

\newcommand\dd{\mathrm{d}}

\begin{document}
	

\begin{center}
	{\Large
		{\sc  Spectral estimation of Hawkes processes from count data}
	}
	\bigskip
	
	Felix Cheysson $^{1,2,3}$ \& Gabriel Lang $^{1}$
\end{center}
\bigskip

{\it
	\noindent $^{1}$ UMR MIA-Paris, Université Paris-Saclay, AgroParisTech, INRAE, Paris, France.\\
	\noindent $^{2}$ Epidemiology and Modeling of bacterial Evasion to Antibacterials Unit (EMEA), Institut Pasteur, Paris, France.\\
	\noindent $^{3}$ Anti-infective Evasion and Pharmacoepidemiology Team, Centre for Epidemiology and Public health (CESP), Inserm / UVSQ, France.
}

\bigskip


{\bf Abstract.} 
This paper presents a parametric estimation method for ill-observed linear stationary Hawkes processes. 
When the exact locations of points are not observed, but only counts over time intervals of fixed size, methods based on the likelihood are not feasible. 
We show that spectral estimation based on Whittle's method is adapted to this case and provides consistent and asymptotically normal estimators, provided a mild moment condition on the reproduction function.
Simulated datasets and a case-study illustrate the performances of the estimation, notably of the reproduction function even when time intervals are relatively large.



\section{Introduction}
Hawkes processes, introduced in \citep{Hawkes1971a, Hawkes1971b}, form a family of models for point processes which exhibit both self-exciting (\textit{i.e.} the occurrence of any event increases temporarily the probability of further events occurring) and clustering properties: they are special cases of the Poisson cluster process, where each cluster is a continuous-time Galton--Watson tree with a Poisson offspring process, the intensity of which is called the reproduction function \citep{Hawkes1974}.
As they exhibit self-exciting and clustering properties, Hawkes processes are appealing in point process modeling, and while first applications concerned almost exclusively seismology \citep{Adamopoulos1976, Ogata1988}, their use quickly spread to many other disciplines, including neurophysiology \citep{Chornoboy1988}, finance \citep{Bowsher2003, Chavez-Demoulin2005, Bacry2015}, genomics \citep{Reynaud-Bouret2010} and epidemiology \citep{Meyer2012}; see also \citet{Reinhart2018} for a review of Hawkes processes and their applications.

Parameter estimation of Hawkes processes has been studied thoroughly when events are fully observed, relying mainly on maximum likelihood methods \citep{Ogata1978, Ozaki1979, Ogata1988}.
Here we consider that the arrival times are not observed, but interval censored: the timeline is cut into regular bins corresponding to {\it e.g.} days or weeks and the numbers of events in each bin is counted. 
Exact maximum likelihood methods are no more applicable to such bin-count data: since the resulting process is no longer a point process but a time series, either the point process must be reconstructed from the count data, or the estimation method of the process must be adapted to time series.

For the first strategy, one could assign arbitrarily to each point a location within its interval, {\it e.g.} by uniformly drawing them in the interval \citep{Meyer2012}, or attempt a more sophisticated approach such as an expectation maximization algorithm.
Historically for Hawkes processes, this algorithm has been used for multivariate processes \citep{Olson2013} or when the immigration intensity is a renewal process \citep{Wheatley2016}, treating the genealogy as a latent variable.
For an interval censored process, an analogous approach which would consider the arrival times as latent variables is unfortunately not adapted, since there is no closed form for the conditional distribution of the arrival times given the event counts.
Stochastic expectation maximization algorithms \citep{Celeux1995}, which approximate this conditional distribution, do not alleviate this issue since usual convergence results and simulation methods are based on likelihoods of the exponential families \citep{Delyon1999}, which excludes Hawkes processes.

For the second strategy, \citet{Kirchner2016} showed that the distribution of the bin-count sequence of the Hawkes process can be approximated by an INAR($\infty$) sequence and proposed a non-parametric estimation method for the Hawkes process.
In particular, the conditional least-square estimates of the INAR process yields consistent and asymptotically normal estimates for the underlying Hawkes process when the bin size tends to zero \citep{Kirchner2017}.
Unfortunately, while this method is adapted when the bin size can be chosen arbitrarily small, for example when the data are collected continuously (seismology, finance, \textit{etc.}), it is biased when the data has been collected with large bin size or when the events cannot precisely be located in time \citep{Kirchner2017}, as is often the case for biological, ecological and health datasets.
In particular when the bin size is larger than the typical range of the reproduction function, this strategy is not satisfactory, since the INAR model ignores the interaction within bins. 

In this article, as in \citet{Kirchner2017}, we adapt an existing time series estimation method to the case of bin-count sequences from Hawkes processes.
Following \citet{Adamopoulos1976}, we use the Bartlett spectrum of the Hawkes process (\textit{i.e.} the spectral density of the covariance measure of the process) to define as an estimator the minimiser of the log-spectral likelihood, first introduced by \citet{Whittle1952}.
To establish the asymptotic properties of the Whittle estimator, we look at strong mixing properties for the Hawkes processes.

\citet{Rosenblatt1956} introduced the strong mixing coefficient to measure the dependence between $\sigma$-algebras, which sparked decades of interest in the theory of weak dependence for time series and random fields (see \citealp{Bradley2005} for a review of mixing conditions). 
The mixing conditions provide very strong inequalities and coupling methods \citep{Doukhan1994, Rio2000} to achieve proofs of asymptotic properties for parameter estimates, provided that the mixing coefficients decrease fast enough.
However, these coefficients are formulated with respect to rich $\sigma$-algebras and therefore difficult to bound even for very simple models.

\citet{Westcott1972} extended the definition of mixing to point processes, proving for example that cluster Poisson point processes are mixing in the ergodic sense \citep{Westcott1971}.
Yet, without precise information on strong mixing coefficients, the weak dependence framework did not lead to much statistical development in the modeling of point processes.
Recent works addressed the computation of strong mixing coefficients for some classes of point processes \citep{Heinrich2013, Poinas2017}, building on the results for time series and random fields and using the fact that the $\sigma$-algebras generated by countable sets are poorer than those generated by continuous sets.

For practical reason, the absolute regularity mixing coefficients are often preferred since they can be easily computed for Markov processes and functions thereof \citep{Davydov1974}. 
Notably, Hawkes processes with exponential reproduction function are piecewise deterministic Markovian processes \citep{Oakes1975}, and one would hope to compute absolute regularity mixing coefficients.
However, since this would not extend to other reproduction functions, we instead establish a strong mixing condition with polynomial decay rate which holds for any reproduction function, provided it has a finite moment of order $1 + \delta$, $\delta > 0$.
In turn, this proves that our proposed estimation method leads to consistent and asymptotically normal estimators for the parameters of Hawkes processes from bin-count data.

Section \ref{sec:first} recalls definitions and sets notations used in the paper.
Section \ref{sec:alphamixing} contains our first important result: we establish strong mixing properties for the Hawkes process and its bin-count sequences.
Using the cluster and positive association properties, we relate the strong mixing coefficients to those of a single time-continuous Galton--Watson tree, then control the covariance between arrival times using results from elementary Galton--Watson theory.
In Section \ref{sec:estimation}, we focus on the estimation of Hawkes processes from bin-count data.
We derive the spectral density function of the bin-count sequence, taking into account the aliasing caused by sampling the process in discrete time.
Then, using the strong mixing condition and the work of \citet{Dzhaparidze1986} on Whittle's method, we propose a consistent and asymptotically normal estimator to the parameters of the Hawkes process.
Sections \ref{SEC:SIMULATIONS} and \ref{sec:application} provide respectively some numerical experiments and a real-life case-study to illustrate the results of the two preceding sections.
Finally, in Section \ref{sec:conclusion}, we discuss some of the appealing features and extensions of this approach.
The code used in the paper, both for the simulation- and the case-study, is publicly available (\url{https://github.com/fcheysson/code-spectral-hawkes}).

\section{The Hawkes process and its count process}
\label{sec:first}
\subsection{Notation}
In this paper, we consider \emph{simple locally finite point processes} on the measure space $(\mathbb R, \mathcal B(\mathbb R), \ell)$, where $\mathcal B(A)$ denotes the Borel $\sigma$-algebra of $A$ and $\ell$ the Lebesgue measure.
A point process $N$ on $\mathbb R$ may be defined as a measurable map from a probability space $(\mathcal{X}, \mathcal{F}, \mathbb{P})$ to the measurable space $(\mathfrak N, \mathcal N)$ of locally finite counting measures on $\mathbb R$.
The corresponding random set of points, \textit{i.e.} the atoms of $N$, is denoted $\{T_i\}$.
For a function $f$ on $\mathbb R$, we write 
\begin{equation*}
N(f) \coloneqq \int_\mathbb R f(t) N(\dd t) = \sum_i f(T_i)
\end{equation*}
the integral of $f$ with respect to $N$.
Finally, for a Borel set $A$, the cylindrical $\sigma$-algebra $\mathcal E(A)$ generated by $N$ on $A$ is defined by
\begin{equation*}
\mathcal E(A) \coloneqq \sigma\big(\{N \in \mathfrak N: N(B) = m\}, B \in \mathcal B(A), m \in \mathbb N\big).
\end{equation*}

\subsection{The stationary linear Hawkes process}
\label{sec:branch}
A stationary self-exciting point process, or \emph{Hawkes process}, on the real line $\mathbb{R}$ is a point process $N$ with conditional intensity function 
\begin{align*}
\lambda(t) &= \eta + \int_0^t h(t-u) N(\dd u) \\
&= \eta + \sum_{T_i < t} h(t - T_i)
\end{align*}
for $t \in \mathbb R$.
The constant $\eta > 0$ is called the \emph{immigration intensity} and the measurable function $h : \mathbb{R}_{\ge 0} \rightarrow \mathbb{R}_{\ge 0}$ the \emph{reproduction function}. 
The reproduction function can be further decomposed as $h = \mu h^\ast$, where $\mu = \int_\mathbb R h(t) \dd t < 1$ is called the \emph{reproduction mean} and $h^\ast$ is a true density function, $\int_\mathbb{R} h^\ast(t) \dd t = 1$, called the \emph{reproduction kernel}.

Moreover, the linear Hawkes process is a specific case of the Poisson \emph{cluster process} \citep{Hawkes1974}.
Briefly, the process consists of a stream of \emph{immigrants}, the cluster centres, which arrive according to a Poisson process $N_c$ with intensity measure $\eta$. 
Then, an immigrant at time $T_i$ generates \emph{offsprings} according to an inhomogenous Poisson process $N_1(\cdot \vert T_i)$ with intensity measure $h(\cdot - T_i)$. 
These in turn independently generate further offsprings according to the same law, and so on \textit{ad infinitum}.
The \emph{branching processes} $N(\cdot \vert T_i)$, consisting of an immigrant at time $T_i$ and all their \emph{descendants}, are therefore independent.
Finally, the Hawkes process $N$ is defined as the superposition of all branching processes:
\begin{equation*}
\forall A \in \mathcal{B}(\mathbb{R}), N(A) = N_c \big( N(A \vert \cdot) \big).
\end{equation*}

This cluster representation links to the usual Galton--Watson theory. 
Without loss of generality, consider one branching process whose immigrant has time $0$.
Define $Z_k$ as the number of points of generation $k$, \textit{i.e.} $Z_0 = 1$ for the immigrant, then $Z_1$ denotes the number of offsprings that the immigrant generates, $Z_2$ the number of offsprings that the offsprings of the immigrants generate, \textit{etc.}
Then $(Z_k)_{k \in \mathbb{N}}$ is a Galton--Watson process.

In particular, $(Z_{k+1} \,\vert\, Z_k = z) ~ (k, z \in \mathbb{N})$ follows a Poisson distribution with parameter $z\mu$.
Then, by the usual Galton--Watson theory, a sufficient condition for the existence of the Hawkes process is $\mu < 1$ which ensures that the total number of descendants of any immigrant is finite with probability $1$ and has finite mean.
This condition also ensures that the process is strictly stationary.

\subsection{Count processes}
We are interested in the time series generated by the event counts of the Hawkes process, that is the series obtained by counting the events of the process on intervals of fixed length.
We give a definition for both time-continuous and discrete time bin-count processes, according to whether the interval endpoints live on the real line or on a regular grid respectively:
\begin{definition}
The bin-count process with bin size $\Delta$ associated to a point process $N$ is the process $(X_t)_{t \in \mathbb{R}} = \big\{ N\big((t \Delta, (t+1) \Delta]\big)\big\}_{t \in \mathbb{R}}$ generated by the count measure on intervals of size $\Delta$.
The restriction of the bin-count process on $\mathbb Z$, $(X_k)_{k \in \mathbb{Z}}$, is called the bin-count sequence associated to $N$.
\end{definition}

\section{Strong mixing properties}
\label{sec:alphamixing}
Here, we control the strong mixing coefficients of Hawkes processes and their associated bin-count processes.
We recall that, for a probability space $(\mathcal{X}, \mathcal{F}, \mathbb{P})$ and $\mathcal{A}, \mathcal{B}$ two sub $\sigma$-algebras of $\mathcal{F}$, Rosenblatt's strong mixing coefficient is defined as the measure of dependence between $\mathcal{A}$ and $\mathcal{B}$ \citep{Rosenblatt1956}:
\begin{equation*}
\alpha(\mathcal{A}, \mathcal{B}) \coloneqq \mathrm{sup} \big\{ \left| \mathbb{P}(A \cap B) - \mathbb{P}(A) \mathbb{P}(B) \right| : A \in \mathcal{A}, B \in \mathcal{B} \big\}.
\end{equation*}
This definition can be adapted to a point process $N$ on $\mathbb{R}$, by defining (see \citealp{Poinas2017})
\begin{equation*}
\alpha_N(r) \coloneqq \underset{t \in \mathbb R}{\mathrm{sup}} 
~\alpha\big(\mathcal E_{-\infty}^t, \mathcal E_{t+r}^\infty\big),
\end{equation*}
where $\mathcal E_a^b$ stands for $\mathcal E\big((a,b]\big)$, \textit{i.e.} the $\sigma$-algebra generated by the cylinder sets on the interval $(a, b]$, and $\mathcal E_a^\infty = \sigma (\cup_{b > a} \mathcal E_a^b)$.
For the corresponding sequence $(X_k)_{k \in \mathbb{Z}}$, the strong mixing coefficient takes the form
\begin{equation*}
\alpha_X(r) \coloneqq \underset{n \in \mathbb{Z}}{\mathrm{sup}}~\alpha \big( \mathcal{F}_{-\infty}^n, \mathcal{F}_{n+r}^\infty \big),
\end{equation*}
where $\mathcal{F}_a^b$ stands for the $\sigma$-algebra generated by $\{X_k:k\in \mathbb Z, a \le k \le b\}$.

The point process $N$ (resp. the sequence $(X_k)$) is said to be strongly mixing if $\alpha_N(r)$ (resp. $\alpha_X(r)$) $\rightarrow 0$ as $r \rightarrow \infty$.
Intuitively, the strong mixing condition conveys that the dependence between past and future events  decreases uniformly to zero as the time gap between them increases.
Note that, since $\mathcal F_a^b \subset \mathcal E\big((a \Delta,(b+1)\Delta]\big)$, we have that $\alpha_X(r) \le \alpha_N((r-1)\Delta)$ for all $r>1$.

We here state the first important result of this article:
\begin{theorem}
\label{THM:ALPHAMIXING}
Let $N$ be a stationary
Hawkes process on $\mathbb{R}$ 
with reproduction function $h = \mu h^\ast$, where $\mu = \int_\mathbb R h(t) \dd t < 1$ and $\int_\mathbb{R} h^\ast(t) \dd t = 1$.
Suppose that there exists $\delta > 0$ such that the reproduction kernel $h^\ast$ has a finite moment of order $1+\delta$:
\begin{equation*}
\nu_{1+\delta} \coloneqq \int_\mathbb{R} t^{1+\delta} h^\ast(t)\dd t < \infty.
\end{equation*}
Then $N$ is strongly mixing and 
\begin{equation*}
\alpha_N(r) = \mathcal{O}\big(r^{-\delta}\big).
\end{equation*}
Furthermore, if $h^\ast$ admits finite exponential moments, that is, there exists $a_0 > 0$ such that 
\begin{equation*}
    \int_{\mathbb R} e^{a_0 |t|} h^\ast(t) \dd t < \infty,
\end{equation*}
then there exists $a \in (0, a_0]$ such that
\begin{equation*}
\alpha_N(r) = \mathcal{O}\big(e^{-ar}\big).
\end{equation*}
\end{theorem}

Notably, this theorem can be extended in the case where the immigration intensity $\eta$ is allowed to vary with respect to time:
\begin{corollary}
\label{COR:NONSTATIONARY}
Let $N$ be a Hawkes process on $\mathbb R$ with reproduction function as in Theorem \ref{THM:ALPHAMIXING}, and with non constant immigration intensity $\eta: t \mapsto \eta(t)$.
Suppose that there exists $\delta > 0$ such that the reproduction kernel $h^\ast$ has a finite moment of order $1+\delta$, and that $\eta(\cdot)$ is non-negative and bounded.
Then $N$ is strongly mixing and 
\begin{equation*}
\alpha_N(r) = \mathcal{O}\big(r^{-\delta}\big).
\end{equation*}
Furthermore, if $h^\ast$ admits finite exponential moments, then there exists $a > 0$ such that
\begin{equation*}
\alpha_N(r) = \mathcal{O}\big(e^{-ar}\big).
\end{equation*}
\end{corollary}

In brief, the proof has two parts: first, we rescale the problem to a single continuous-time Galton--Watson tree using the cluster representation of the Hawkes process; second, we derive an upper bound for the strong mixing coefficients of the tree.
The idea for the latter is that since the Galton--Watson process goes extinct almost surely and the reproduction kernel $h^\ast$ has a finite moment, then the probability that there exists an offspring of generation $k$ at a far distance from the immigrant goes quickly to 0 when $k$ increases.
We refer to Appendix \ref{sec:proof} for the detailed proof of the theorem.

Finally, as an immediate consequence of Theorem \ref{THM:ALPHAMIXING}, we get the following corollary for Hawkes bin-count process:
\begin{corollary}
\label{cor:alphamixing}
Let $N$ be a Hawkes process as in Corollary \ref{COR:NONSTATIONARY}, and $(X_k)_{k \in \mathbb{Z}} = \big\{ N\big((k \Delta, (k+1) \Delta]\big)\big\}_{k \in \mathbb{Z}}$ its associated bin-count sequence. Then $(X_k)$ is strongly mixing and 
\begin{equation}
\label{eqn:alphamixing}
\alpha_X(r) = \mathcal{O}\big(r^{-\delta}\big).
\end{equation}
Furthermore, if $h^\ast$ admits finite exponential moments, then there exists $a > 0$ such that
\begin{equation*}
\alpha_X(r) = \mathcal{O}\big(e^{-a\Delta r}\big).
\end{equation*}
\end{corollary}

\section{Parametric estimation of bin-count sequences}
\label{sec:estimation}
In this section, we apply the strong mixing properties of the Hawkes bin-count sequence to parametric estimation using a spectral approach.
First, we derive the spectral density function for both the time-continuous and discrete time Hawkes bin-count processes.
Then using Whittle's method, we define a parametric estimator of a Hawkes process from its bin-count data.

\subsection{Spectral analysis}
We recall that the \emph{Bartlett spectrum} of a second order stationary point process $N$ on $\mathbb{R}$ is defined as the unique, non-negative, symmetric measure $\Gamma$ on the Borel sets such that, for any rapidly decaying function $\varphi$ on $\mathbb{R}$ (see \citealp[Proposition~8.2.I, equation (8.2.2)]{Daley2003})
\begin{equation}
\label{eqn:varN}
\mathrm{Var}\big(N(\varphi)\big) = \int_{\mathbb R} \vert \widetilde{\varphi}(\omega) \vert^2 \Gamma(\dd\omega),
\end{equation}
where $\widetilde{\cdot}$ denotes the Fourier transform: 
\begin{equation*}
\widetilde{\varphi}(\omega) = \int_\mathbb{R} \varphi(s) e^{-i\omega s} \dd s.
\end{equation*}
By polarising relation \eqref{eqn:varN}, we get, for any rapidly decaying functions $\varphi$ and $\psi$ on $\mathbb{R}$:
\begin{equation}
\label{eqn:bartlett}
\mathrm{Cov}\big(N(\varphi), N(\psi)\big) = \int_\mathbb{R} \widetilde\varphi(\omega) \widetilde{\psi^\ast}(\omega) \Gamma (\dd\omega),
\end{equation}
where $\psi^\ast(u) = \psi(-u)$, so that $\widetilde{\psi^\ast}$ is the complex conjugate of $\widetilde{\psi}$.

For the stationary Hawkes process, the Bartlett spectrum admits a density given by (see \citealp[Example~8.2(e)]{Daley2003})
\begin{equation}
\label{eqn:hbartlett}
\gamma(\omega) = \frac{m}{2 \pi} \left| 1 - \widetilde{h}(\omega) \right|^{-2}
\end{equation}
where $m \coloneqq \mathbb{E}\big[ N(0,1] \big] = \eta \left(1 - \int_\mathbb{R} h(t)dt\right)^{-1}$.

For a time-continuous process, the spectral density $f_c$ forms a Fourier pair with the autocovariance function $r_c$:
\begin{equation*}
f_c(\omega) = \int_{\mathbb R} r_c(u) e^{-i\omega u} \dd u, \qquad r_c(u) = \frac{1}{2\pi} \int_{\mathbb R} f_c(\omega) e^{i\omega u} \dd \omega.
\end{equation*}
Using the second relation with the Bartlett spectrum of a stationary Hawkes process, we derive the spectral density of the time-continuous bin-count process with bin size $\Delta$:
\begin{proposition}
Let $N$ be a stationary Hawkes process on $\mathbb{R}$, and $\{X_t\}_{t \in \mathbb{R}} = \left\{ N(t \Delta, (t+1) \Delta] \right\}_{t \in \mathbb{R}}$ the associated bin-count process.
Then $X_t$ has a spectral density function given by
\begin{equation}
\label{eqn:hspec}
f_c(\omega) = m\,\Delta\,\mathrm{sinc}^2 \left( \frac{\omega}{2} \right) \left| 1 - \widetilde{h} \left( \frac{\omega}{\Delta} \right) \right|^{-2}.
\end{equation}
\end{proposition}

\begin{proof}
Let $\varphi = \mathbbm{1}_{(0, \Delta]}$ and $\psi = \mathbbm{1}_{(\Delta u, \Delta(u+1)]}$.
We have 
\begin{align*}
\widetilde\varphi(\omega) &= \int_0^\Delta e^{-i\omega s} \dd s = \frac{i}{\omega} \left[e^{-i\omega \Delta}-1\right],\\
\widetilde{\psi^\ast}(\omega) &= \int_{-\Delta(u+1)}^{-\Delta u} e^{-i\omega s} \dd s = \frac{i}{\omega} e^{i\omega \Delta u} \left[1 - e^{i\omega \Delta}\right].
\end{align*}
Then, using (\ref{eqn:bartlett}) and (\ref{eqn:hbartlett}), the autocovariance function of $X_t$ is
\begin{align*}
r_c(u) &= \mathrm{Cov}(X_0, X_u)\\
&=\mathrm{Cov}\big(N(\varphi), N(\psi)\big)\\
&= \int_\mathbb{R} \frac{1}{\omega^2}\, e^{i\omega\Delta u} \left|e^{i\omega\Delta}-1\right|^2 \,\Gamma(\dd\omega)\\
&= \frac{1}{2\pi} \int_\mathbb{R} m\Delta \,\mathrm{sinc}^2 \left(\frac{\omega}{2}\right) \left| 1-\widetilde{h}\left(\frac{\omega}{\Delta}\right)\right|^{-2} e^{i\omega u} \dd\omega.
\end{align*}
\end{proof}

For a discrete-time process, the spectral density $f_d$ once again forms a Fourier pair with the autocovariance function $r_d$:
\begin{equation*}
f_d(\omega) = \sum_{k \in \mathbb Z} r_d(u) e^{-i \omega u}, \qquad r_d(u) = \frac{1}{2\pi} \int_{-\pi}^{\pi} f_d(\omega) e^{i \omega u} \dd \omega.
\end{equation*}
It turns out that the spectral density $f_d$ of a time-continuous process sampled in discrete time can be related to the density  $f_c$ of the process, by taking into account spectral aliasing, which folds high frequencies back onto the apparent spectrum:
\begin{align*}
r_c(u)
&= \frac{1}{2\pi} \int_\mathbb{R} f_c(\omega) e^{i\omega u} \dd\omega\\
&= \frac{1}{2\pi} \sum_{k \in \mathbb Z} \int_{(2k-1)\pi}^{(2k+1)\pi} f_c(\omega) e^{i\omega u} \dd\omega\\
&= \frac{1}{2\pi} \sum_{k \in \mathbb Z} \int_{-\pi}^{\pi} e^{i2k\pi u} f_c(\omega +2k\pi) e^{i\omega u} \dd\omega\\
&= \frac{1}{2\pi} \int_{-\pi}^{\pi} \sum_{k \in \mathbb Z} f_c(\omega + 2k\pi) e^{i\omega u} \dd\omega,
\end{align*}
where the last equality follows when $u \in \mathbb Z$ and from an application of Fubini's theorem.

For the bin-count sequence associated with a stationary Hawkes process, this leads to the following corollary:
\begin{corollary}
Let $N$ be a stationary Hawkes process on $\mathbb{R}$, and $(X_k)_{k \in \mathbb{Z}} = \big\{ N\big((k \Delta, (k+1) \Delta] \big)\big\}_{k \in \mathbb{Z}}$ the associated bin-count sequence.
Then $X_k$ has a spectral density function given by
\begin{equation*}
f_d(\omega) = \sum_{k \in \mathbb{Z}} f_c(\omega+2k\pi)
\end{equation*}
where $f_c(\cdot)$ is the function defined in (\ref{eqn:hspec}).
\end{corollary}

\subsection{Whittle estimation}
\label{sec:whittle}
For a stationary linear process $(X_k)_{k \in \mathbb Z}$ with spectral density $f_\theta(\cdot)$, $\theta$ an unknown parameter vector, both \citet{Hosoya1974} and \citet{Dzhaparidze1974}, building on the cornerstone laid by \citet{Whittle1952}, proposed as an estimator of $\theta$ the minimizer
\begin{equation}
\label{eqn:estimator}
\widehat \theta_n = \mathrm{arg}~\underset{\theta \in \Theta}{\mathrm{min}} ~\mathcal L_n(\theta)
\end{equation}
where
\begin{equation}
\label{eqn:wlik}
\mathcal{L}_n(\theta) = \frac{1}{4\pi} \int_{-\pi}^\pi \left( \log f_\theta(\omega) + \frac{I_n(\omega)}{f_\theta(\omega)} \right)\dd\omega
\end{equation}
is the log-spectral likelihood of the process, and $I_n(\omega) = (2\pi n)^{-1} \left|\sum_{k=1}^n X_k \, e^{-ik\omega}\right|^2$ is the periodogram of the partial realisation $(X_k)_{1 \le k \le n}$.
They also gave the asymptotic properties of the estimator under appropriate regularity conditions.

\citet{Dzhaparidze1986} extended these results to more general cases, and in particular to stationary processes verifying Rosenblatt's mixing conditions.
The following conditions and theorems are thus adaptations of those found in \citet[Theorem~II.7.1~and~II.7.2]{Dzhaparidze1986} for stationary Hawkes bin-count sequences.

\begin{theorem}
\label{thm:consistent}
Let $N$ be a Hawkes process on $\mathbb{R}$ with reproduction function $h = \mu h^\ast$, where $\mu = \int_\mathbb R h(t) \dd t < 1$ and $\int_\mathbb{R} h^\ast(t) \dd t = 1$, and $(X_k)_{k \in \mathbb{Z}} = \big(N(k\Delta, (k+1)\Delta]\big)_{k \in \mathbb{Z}}$ its associated bin-count sequences with spectral density function $f_\theta$.
Assume the following regularity conditions on $f_\theta$:
\begin{itemize}
\itemsep0em 
\item[(A1)] The true parameter $\theta_0$ belongs to a compact set $\Theta$ of $\mathbb R^p$.
\item[(A2)] For all $\theta_1 \ne \theta_2$ in $\Theta$, then $f_{\theta_1} \ne f_{\theta_2}$ for almost all $\omega$.
\item[(A3)] The function $f_\theta^{-1}$ is differentiable with respect to $\theta$ and its derivatives $(\partial / \partial \theta_k)f_\theta^{-1}$ are continuous in $\theta \in \Theta$ and $-\pi \le \omega \le \pi$.
\end{itemize}
Further assume that there exists a $\delta > 0$ such that the reproduction kernel $h^\ast$ has a finite moment of order $2+\delta$.
Then the estimator $\widehat \theta_n$ defined as in (\ref{eqn:estimator}) (with $\mathcal L_n(\theta)$ given by (\ref{eqn:wlik})), is consistent, \textit{i.e.} $\widehat\theta_n \rightarrow \theta_0$ in probability.
\end{theorem}
\begin{proof}
The only condition from \citet[Theorem~II.7.1]{Dzhaparidze1986} that we need to verify is that there exists a $\gamma > 2$ such that $\mathbb E[|X_k|^{2\gamma}]$ is finite and the following inequality holds:
\begin{equation}
\label{eqn:alphafinite}
\sum_{r=1}^\infty \big(\alpha_X(r)\big)^{1-2/\gamma} < \infty.
\end{equation}
Since the stationary Hawkes process admit finite exponential moments if $h^\ast$ has a moment of order $\delta \in (0, 1]$ \citep[Theorem 4]{Roueff2016}, $\mathbb E[|X_k|^{2\gamma}]$ is finite for any $\gamma$.
Then using Equation \eqref{eqn:alphamixing} from Corollary \ref{cor:alphamixing} there always exists a $\gamma > 2$ that satisfies (\ref{eqn:alphafinite}).
\end{proof}

Define the matrix $\Gamma_\theta$, which would actually be the limit as $n \rightarrow \infty$ of the Fisher's information matrix if the process $(X_k)$ were Gaussian \citep[Section II.2.2]{Dzhaparidze1986}, by the relation:
\begin{equation*}
\Gamma_\theta = \left( \frac{1}{4\pi} \int_{-\pi}^\pi \frac{\partial}{\partial\theta_k}\mathrm{log}\,f_\theta(\omega)\, \frac{\partial}{\partial\theta_l}\mathrm{log}\,f_\theta(\omega)\,\dd\omega \right)_{1 \le k,l \le p}.
\end{equation*}
Since $(X_k)$ is not Gaussian, the asymptotic properties of the Whittle estimator depends on the fourth-order statistics of the process and we define the following matrix:
\begin{equation*}
C_{4, \theta} = \left( \frac{1}{8\pi} \int\int_{-\pi}^\pi f_{4,\theta}(\omega_1, -\omega_1, -\omega_2)\  \frac{\partial}{\partial\theta_k}\frac{1}{f_\theta(\omega_1)}\   \frac{\partial}{\partial\theta_l}\frac{1}{f_\theta(\omega_2)}\ \dd\omega_1\dd\omega_2 \right)_{1\le k,l\le p}
\end{equation*}
where $f_{4,\theta}(\cdot, \cdot, \cdot)$ is the fourth-order cumulant spectral density of the process. 
We have the following result:
\begin{theorem}
\label{thm:asymp}
Let $N$ be a Hawkes process as in Theorem \ref{thm:consistent}, and $(X_k)_{k \in \mathbb{Z}} = \big(N(k\Delta, (k+1)\Delta]\big)_{k \in \mathbb{Z}}$ its associated bin-count sequences with spectral density function $f_\theta$.
Assume conditions (A1), (A2), (A3) and:
\begin{itemize}
\item[(A4)] The function $f_\theta$ is twice differentiable with respect to $\theta$ and its second derivatives $(\partial^2 / \partial \theta_k \partial \theta_l)f_\theta$ are continuous in $\theta \in \Theta$ and $-\pi \le \omega \le \pi$.
\end{itemize}
Then the estimator $\widehat\theta_n$ is asymptotically normal and 
\begin{equation*}
n^{1/2}(\widehat\theta_n - \theta_0) \underset{n \rightarrow \infty}\sim \mathcal N\left(0, \Gamma_{\theta_0}^{-1} + \Gamma_{\theta_0}^{-1} C_{4, {\theta_0}} \Gamma_{\theta_0}^{-1} \right).
\end{equation*}
\end{theorem}

\begin{remark}
The computation of the integral of the fourth-order cumulant spectra in $C_{4, {\theta_0}}$ is not straightforward.
We refer to the work of \citet{Shao2010} for an elegant way to compute an estimate of this integral.
\end{remark}

\section{Simulation study}
\label{SEC:SIMULATIONS}
We illustrate the estimation procedure and asymptotic properties of the spectral approach for Hawkes bin-count sequences.
To highlight the different theorems of the previous sections, we consider two kernels $h^\ast$ for the reproduction function: the exponential kernel for which all moments exist and the power law kernel whose higher moments are not finite.

The following simulations and estimations have been implemented with our package \emph{hawkesbow}, freely available online (\url{https://cran.r-project.org/web/packages/hawkesbow/index.html}), written in both \textsc R \citep{RCoreTeam2019} and C++ using Rcpp \citep{Eddelbuettel2013}.

\subsection{Simulation procedure}
\label{sec:simproc}
\subsubsection{Exponential kernel}
We first consider a stationary Hawkes process with exponentially decaying reproduction function:
\begin{equation*}
\lambda(t) = \eta + \mu \int_{-\infty}^t \beta e^{-\beta (t-u)} N(\dd u), \qquad t \in \mathbb R,
\end{equation*}
\textit{i.e.} with reproduction kernel $h^\ast(t) = \beta e^{-\beta t}$ for $t \ge 0$.
Note that the process verifies the conditions of both Theorems \ref{thm:consistent} and \ref{thm:asymp}.

Using the cluster representation of the Hawkes process, we simulated 1,000 realisations of the Hawkes process on the interval $[0, T]$ with parameter values $\eta = 1$, $\mu = 0.5$ and $\beta = 1$ and a burn-in interval $[-100, 0]$.
For each of the simulations, we created four time series by counting the events in bins of size $\Delta = 0.25$, $0.5$, $1$ or $2$ respectively.
We then estimated the parameters $\eta$, $\mu$ and $\beta$ as in Section \ref{sec:whittle} for each of the four time series.
We compared these estimates to the usual maximum likelihood estimates (Figure \ref{fig:exponential_estimates}), for which it is assumed no event lies outside of $[0,T]$.
Since the latter use the full information on the location of events, they are arguably better that any estimate based on the bin-count sequences, and provide a best case scenario for the Whittle estimates when the bin size tends to 0.
Minimisation of the log-spectral likelihood, resp. maximisation of the likelihood, was carried out using a limited-memory BFGS optimisation algorithm with bound constraints \citep{Nocedal1980, Liu1989}---$0 < \mu < 1$ and $\eta, \beta > 0$---available in function \textrm{optim} from \textsc R, resp. in function \textrm{nloptr} from package \textsc{nloptr} \citep{Johnson}.
With an exponential kernel, a set of 1,000 simulations and their Whittle estimation with $T = 1000$ and bin size $\Delta = 1$ takes approximately 2 minutes on a laptop computer with an i5-6300HQ Intel CPU.

\subsubsection{Power law kernel}
We now consider a stationary Hawkes process with a power law reproduction kernel: $h_\gamma^\ast(t) = \gamma a^\gamma (a+t)^{-\gamma - 1}$ for $t \ge 0$.
We recall that the moments of a power law distribution are all finite up to, but not including, the order $\gamma$.
We illustrate the theorems of the previous sections by considering three cases for the shape, with each increasingly satisfying the necessary assumptions: 
\textit{(i)} $\gamma = 0.5$, the process does not satisfy the condition of Theorem \ref{THM:ALPHAMIXING}; 
\textit{(ii)} $\gamma = 1.5$, the process is strongly mixing and satisfies the condition of Theorem \ref{THM:ALPHAMIXING}, but not the assumptions of Theorem \ref{thm:consistent}; 
\textit{(iii)} $\gamma = 2.5$, the process is strongly mixing and satisfies the assumptions of Theorems \ref{THM:ALPHAMIXING}, \ref{thm:consistent} and \ref{thm:asymp}.

As for the exponential kernel, we simulated 1,000 realisations of the Hawkes process for each $\gamma \in \{0.5, 1.5, 2.5\}$, with parameter values $\eta = 1$, $\mu = 0.5$, and scale parameter $a = 1.5$ chosen such that the power law kernel $h_{2.5}^\ast$ and the exponential kernel have the same first-order moment.
For the power law kernels $h_{1.5}^\ast$ and $h_{0.5}^\ast$, we kept the scale parameter the same so that the simulations can be compared with the kernel $h_{2.5}^\ast$.
As for the exponential kernel, the Whittle estimates of $\eta$, $\mu$ and $\gamma$ were compared to the usual maximum likelihood estimates, with the scale parameter kept fixed to its true value $a = 1.5$.
Estimation figures can be found in Appendix \ref{sec:figures}.
The constraint bounds for the optimisation routines were $0 < \mu < 1$ and $\eta, \gamma > 0$.
With a power law kernel, a set of 1,000 simulations and their Whittle estimation with $T = 1000$ and bin size $\Delta = 1$ takes approximately 14 minutes on a laptop computer with an i5-6300HQ Intel CPU.

\begin{figure}[h!]
\centering
\begin{subfigure}[t]{\textwidth}
\centering
\includegraphics[width=\textwidth]{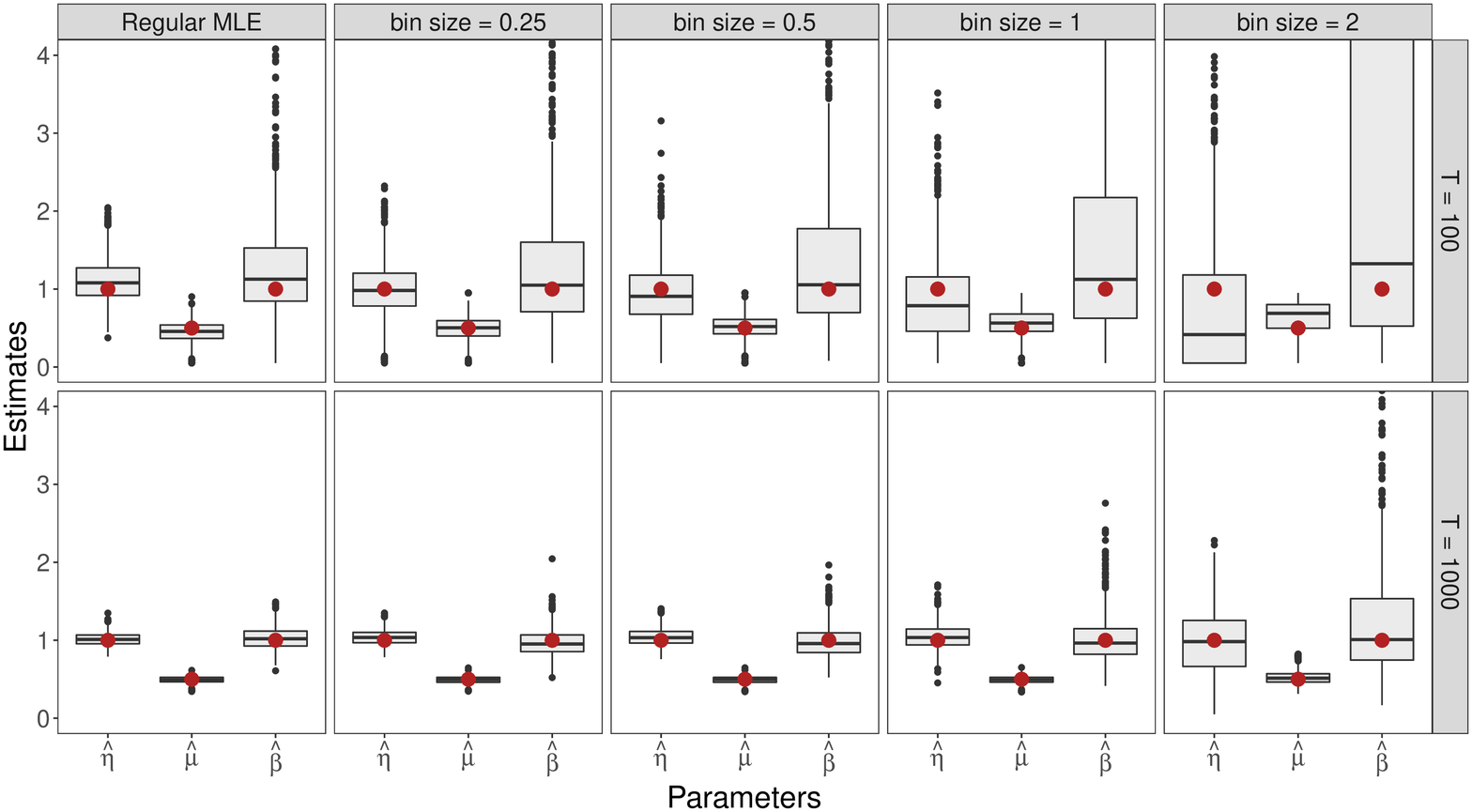}
\caption{Estimates of parameters $\eta$, $\mu$ and $\beta$ for 1,000 simulations on the interval $[0, T]$.
True values (larger dots) are: $\eta = 1$, $\mu = 0.5$, $\beta = 1$.
The left column refers to the maximum likelihood estimates. The other columns refer to the Whittle estimates according to different bin sizes.}
\label{fig:exponential_estimates}
\end{subfigure} 

\vspace{2ex}
\begin{subfigure}[t]{\textwidth}
\centering
\includegraphics[width=\textwidth]{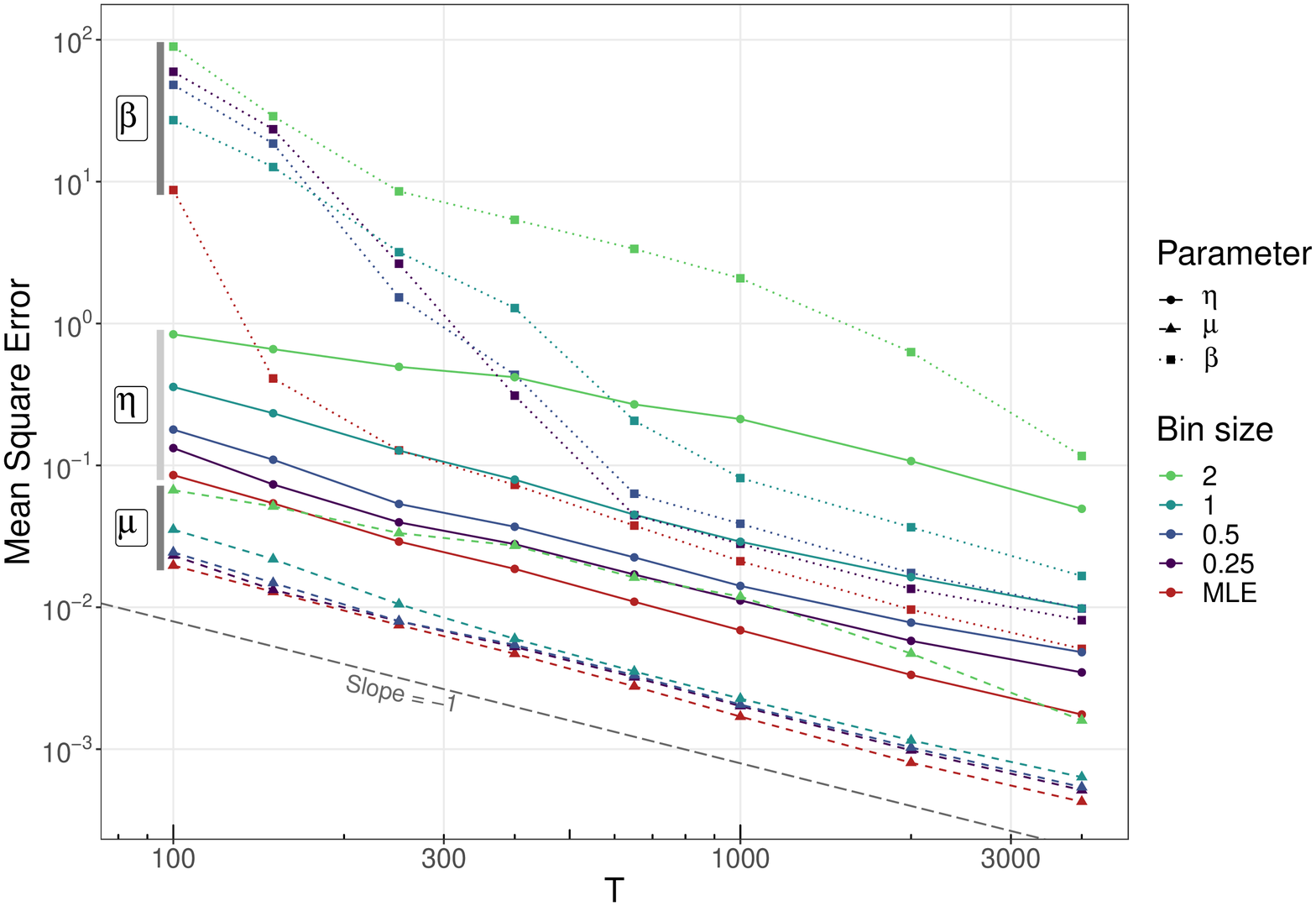}
\caption{Mean square error of the estimates of parameters $\eta$, $\mu$ and $\beta$ for 1,000 simulations on the interval $[0,T]$, in log-log scale.
The dashed grey line represents the ideal slope of $-1$, \textit{i.e.} a rate of convergence of $\mathcal O(n^{-1})$.}
\label{fig:exponential_convergence}
\end{subfigure}
\caption{Performance of the Whittle estimates for the stationary Hawkes process with immigration intensity $\eta = 1$, reproduction mean $\mu = 0.5$, and reproduction kernel $h^\ast(t) = \beta e^{-\beta t}$, where $\beta = 1$.}
\end{figure} 

\subsection{Results and interpretation}
\subsubsection{Exponential kernel}
For $T = 100$ and small bin sizes, the Whittle estimates fare almost as well as the maximum likelihood estimates (see Figure \ref{fig:exponential_estimates}).
The estimation deteriorates massively for higher bin sizes, notably for the exponential kernel rate $\beta$.
This is intuitive, since large bin sizes with respect to the kernel scale make it difficult to detect interactions between points.
This can be related to the probability that an offspring and its parent belong in the same bin: assuming the stationarity of the process, this probability is equal to $\Delta^{-1} \int_0^\Delta \int_u^\Delta \beta e^{-\beta(t-u)} \dd t \dd u = 1 - (\beta \Delta)^{-1} (1 - e^{-\beta \Delta})$.
For example, with $\beta = 1$ and $\Delta = 2$, we get a probability of $0.57$, \textit{i.e.} 57\% of the information about the interaction of the Hawkes process is located within bins, with only 43\% remaining between bins.
Thankfully, by increasing $T$, the asymptotic properties ensure that the Whittle estimates improve, even for large bin sizes.

To further illustrate the asymptotic properties of the estimation, notably its rate of convergence, we compute the mean square error, defined by $\mathrm{MSE} = S^{-1} \sum (\widehat\theta_n - \theta_0)^2$, for the estimates of each set of $S = 1,000$ simulations at given $T$s and bin sizes (Figure \ref{fig:exponential_convergence}).
For large $T$s, the slope of the mean square error with respect to $T$ reaches $-1$ (in log-log scale) for all parameters and almost all bin sizes, illustrating the $\mathcal O(n^{-1})$ rate of convergence stated in Theorem \ref{thm:asymp}.
For small $T$s and both the Whittle and the maximum likelihood estimation methods, the estimates of the immigration intensity $\eta$ and reproduction mean $\mu$ have already reached the optimal rate of convergence, while the MSE for the exponential kernel rate $\beta$ is up to one and a half orders of magnitude higher than what would be expected by extrapolating the MSE for large $T$s.
Finally note that, for reasonable bin sizes ($\Delta\le1$), the Whittle estimates of the reproduction mean $\mu$ have a MSE comparable to those of the maximum likelihood.

\subsubsection{Power law kernel}
Performances for the point estimates are remarkably similar for $\gamma = 2.5$ and $\gamma = 1.5$. 
As in the exponential case, both the immigration intensity $\eta$ and the reproduction mean $\mu$ exhibit the optimal rate of convergence $\mathcal O(n^{-1})$ throughout all $T$s considered for all bin sizes, while the shape parameter $\gamma$ exhibits this asymptotic regime for sufficiently large $T$s ($T \ge 400$).
When $\gamma = 0.5$, the Whittle estimates for both the immigration intensity $\eta$ and the reproduction mean $\mu$ do not considerably improve for the range of $T$s considered, in contrast to the maximum likelihood estimates which approach the optimal asymptotic regime for large $T$s.
For all three kernels, the estimates show a curious behaviour: for large $T$s, almost all estimates for the bin size 0.25 have larger MSE than for bin sizes 0.5 and 1.

Interestingly, the point estimates exhibit good asymptotic behaviours for $\gamma = 1.5$ even though the power law kernel $h_{1.5}^\ast$ does not satisfy the assumptions of Theorems \ref{thm:consistent} and \ref{thm:asymp}, but do not for $\gamma = 0.5$.
This could suggest that the condition on the kernel moments in Theorem \ref{THM:ALPHAMIXING} is too restrictive, and probably lies between 0.5 and 1.5.
Nevertheless, it is mild enough that the spectral approach developed in this article can be useful for applications in many disciplines.

\section{Case-study: transmission of measles in Tokyo}
\label{sec:application}
Measles is a highly contagious viral disease, primarily transmitted via droplets and manifesting as a febrile rash illness.
Despite worldwide efforts to eradicate the disease, it has sprung back in developed countries mainly through imported cases and non vaccinated individuals, generating minor outbreaks.
As the infectious period of measles begins before symptoms are first apparent, in some cases the carriers may be diagnosed after their offsprings.
For this reason, we propose to adapt the Hawkes process to reflect this apparent non causal situation.

\subsection{Extension to a non causal framework}
\label{sec:noncausal}
We consider a natural extension of linear Hawkes processes by relaxing the condition that the reproduction kernel $h^\ast$ has support on $\mathbb R_{\ge 0}$.
Such a process can be defined through the cluster representation presented in Section \ref{sec:branch} by allowing the offsprings to be generated in the past, \textit{i.e.} by allowing $h^\ast$ to take positive values on $\mathbb R_{< 0}$.
We will call this process a non causal Hawkes process, even if its conditional intensity function is intractable.

The results proved in the Section \ref{sec:alphamixing} can be directly extended to non causal Hawkes processes.
Indeed, all proofs but those of Lemmas \ref{lem:onecluster} and \ref{lem:bigO} from Appendix \ref{sec:proof} remain identical.
For Lemma \ref{lem:bigO}, split the integral into two: one from $-\infty$ to $t + r/2$, the other from $t+r/2$ to $+\infty$.
The first integral is treated as written.
For the second integral, Lemma \ref{lem:onecluster} can be adapted using a symmetry argument regarding the location of the immigrant and the interval considered.
In consequence the spectral estimation procedure proposed in Section \ref{sec:estimation} remains applicable, with consistent and asymptotically normal estimators.

\subsection{Estimation of the contagion function}
\label{sec:measles}
In Japan, measles is a notifiable disease: all diagnosed cases must be reported to the government, then investigated to contain potential outbreaks.
The Japanese National Institute of Infectious Diseases publishes weekly reports as well as surveillance data tables for all notifiable diseases (\url{https://www.niid.go.jp/niid/en/survaillance-data-table-english.html}).
We here consider the number of measles cases in the prefecture of Tokyo, from August 2012 to today (Figure \ref{fig:japan_measles}).
We model the weekly count data using a Hawkes process with Gaussian kernel:
\begin{equation*}
h^\ast(t) = \frac{1}{\sigma \sqrt{2\pi}} \,\mathrm{exp} \left( -\frac{(t-\nu)^2}{2\sigma^2} \right), \qquad t \in \mathbb R
\end{equation*}
where $\nu$ can be related to the incubation period and $\sigma$ to the transmission period, then estimate the parameters $\eta$, $\mu$, $\nu$ and $\sigma$ as in Section \ref{sec:whittle}.
We treat the process as stationary because the impact of the seasonality was small compared to local variability.

\begin{figure}[h!]
\includegraphics[width=\textwidth]{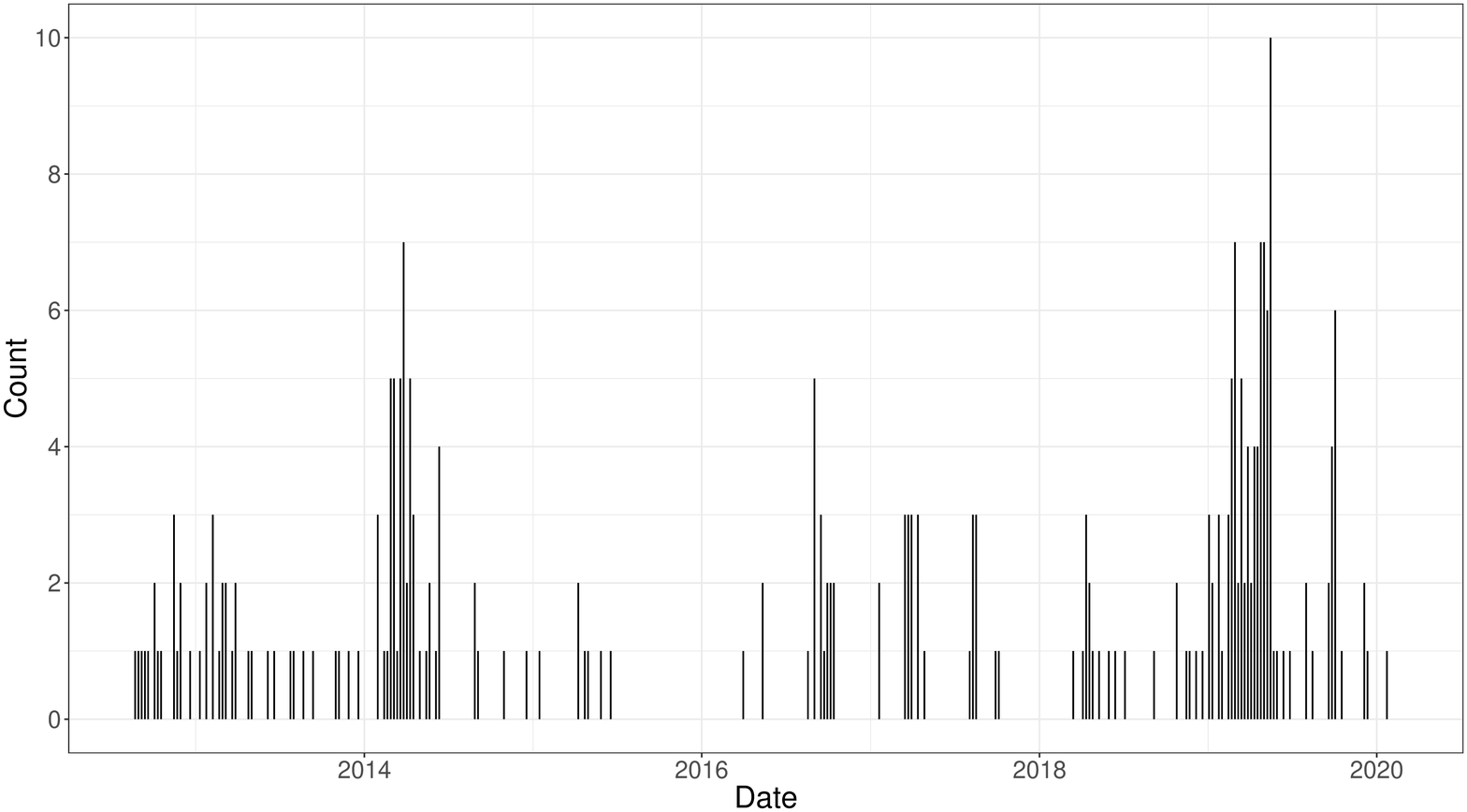}
\caption{Weekly count of measles cases in Tokyo. 
Between the third week of August, 2012 and the third week of February, 2020, 264 cases of measles have been declared in the prefecture of Tokyo.}
\label{fig:japan_measles}
\end{figure} 

For the Gaussian kernel, we find $\widehat\nu = 9.8 ~\mathrm{days}$ and $\widehat\sigma = 5.9 ~\mathrm{days}$, corresponding to an interquartile range of $7.9$ days.
These estimates can be related to clinical features of the virus: the incubation period of measles averages 10-12 days, while the transmission occurs usually from 4 days before to 4 days after rash onset \citep{CDC2015}.
For the immigration intensity and reproduction mean, we find $\widehat\eta = 0.040 ~\mathrm{day}^{-1}$ and $\widehat\mu = 0.72$.
Interestingly, we find that cases with unknown source of transmission (\textit{i.e.} immigrants of the model) represent $1 - \widehat\mu = 28\%$ of all measles cases, a figure close to the data found in \citep[Figure 3]{Nishiura2017}, which reports 23 imported cases amongst 106 contagious events in Japan, 2016.

\subsection{Goodness-of-fit diagnostics}
\label{sec:goodnessoffit}

Assessing the goodness-of-fit of a Hawkes model to the observed data is usually achieved via residual analysis \citep{Ogata1988} where the residuals, which are obtained through an application of the random time change theorem (see \citealp{Papangelou1972} or \citealp[Theorem 7.4.IV]{Daley2003}), are expected to behave like a unitary Poisson point process.
Here, since the arrival times of the process are not observed, we instead use the spectral approach to goodness-of-fit diagnostics for time series models proposed by \citet{Paparoditis2000}.

Using the notations of Subsection \ref{sec:whittle}, the test is based on the distance between a kernel estimator of the normalised periodogram ordinates,
\begin{equation}
\label{eqn:q}
    \widehat q(\omega, \widehat\theta) = \frac{1}{nh} \sum_{j=-m}^m K\left( \frac{\omega - \omega_j}{h} \right) \frac{I_n(\omega_j)}{f_{\widehat\theta}(\omega_j)},
\end{equation}
and its expected value under the null hypothesis, leading to the test statistic given by
\begin{equation*}
    S_{n,h}(\widehat\theta) = n h^{1/2} \int_{-\pi}^\pi \left( \frac{1}{nh} \sum_{j=-m}^m K\left( \frac{\omega - \omega_j}{h} \right) \left( \frac{I_n(\omega_j)}{f_{\widehat\theta}(\omega_j)} - 1 \right) \right)^2 \dd\omega,
\end{equation*}
where $\omega_j = 2\pi j/n$ are the studied frequencies, $m = \lfloor (n-1)/2 \rfloor$, $K$ denotes the kernel and $h$ the bandwidth used to smooth the rescaled periodogram ordinates.
Then, under some regularity assumptions on the studied process $(X_k)$ and the kernel $K$, as $n \rightarrow \infty$ and $h \sim n^{-\rho}$ for some $0 < \rho < 1$ \citep[Theorem 2]{Paparoditis2000},
\begin{equation*}
    S_{n,h}(\widehat\theta) - \mu_h \rightarrow \mathcal N(0, \tau^2),
\end{equation*}
where 
\begin{equation*}
    \mu_h = h^{-1/2} \int_{-\pi}^\pi K^2(x)\dd x \qquad \text{and} \qquad \tau^2 = \frac{1}{\pi} \int_{-2\pi}^{2\pi} \left[ \int_{-\pi}^\pi K(u) K(u+x)\dd u\right]^2 \dd x.
\end{equation*}
Then the null hypothesis, i.e. that the true density function of the process lies in the postulated class of density functions,
\begin{equation*}
    \mathcal H_0: f \in \mathcal F_\Theta = \{ f_\theta, \theta \in \Theta\},
\end{equation*}
can be rejected at an asymptotical $\alpha$-level if $S_{n,h}(\widehat\theta) > \mu_h + u_{1-\alpha}\, \tau$, where $u_{1-\alpha}$ is the $(1 - \alpha)$-th quantile of the standard normal distribution.

For the measles dataset, we calculated the test statistic using the Epanechnikov kernel given by $K(x) = 3(1 - (x/\pi)^2)/2$ for $|x| \le \pi$.
With this choice of kernel, $\mu_h = (12\pi/5)h^{-1/2}$ and $\tau^2 = 2672\pi^2/385$.
This leads to the asymptotic \textit{p}-values $p = 0.61$ and $p = 0.96$ for the bandwidths $h = 0.05$ and $h = 0.10$ respectively.
Bootstrap approximation of the distribution of the test statistic under the null \citep[Section 4]{Paparoditis2000} using 1000 replicates yields similar $p$-values: $p = 0.55$ and $p = 0.97$ respectively.
Hence the chosen Hawkes model seems to correctly reproduce the spectral characteristics of the data (Figure \ref{fig:goodnessOfFit1}).

Additionally, it is possible to derive a goodness-of-fit diagnostic plot which gives helpful information were the postulated model to be rejected, by looking at the asymptotic behaviour of the statistic $\widehat q$ given in \eqref{eqn:q}.
Indeed, under the null \citep[Section 5]{Paparoditis2000},
\begin{equation*}
    Q^2(\omega, \widehat\theta) = \frac{nh(\widehat q(\omega, \widehat\theta) - s_h(\omega))^2}{\frac{1}{2\pi} \int_{-\pi}^\pi K^2(u) \dd u} \rightarrow \chi_1^2,
\end{equation*}
where $s_h(\omega) = (nh)^{-1} \sum_{j=-m}^m K((\omega - \omega_j)/h)$.
Then, a plot of the test statistic $Q^2(\cdot, \widehat\theta)$ can be used to diagnose the frequencies at which the fit of the model must be re-evaluated by comparing the values of $Q^2(\omega, \widehat\theta)$ against the $(1 - \alpha)$-th quantile of the $\chi_1^2$ distribution (Figure \ref{fig:goodnessOfFit2}).

\begin{figure}[h!]
\centering
\begin{subfigure}[t]{\textwidth}
\centering
\includegraphics[width=.9\textwidth]{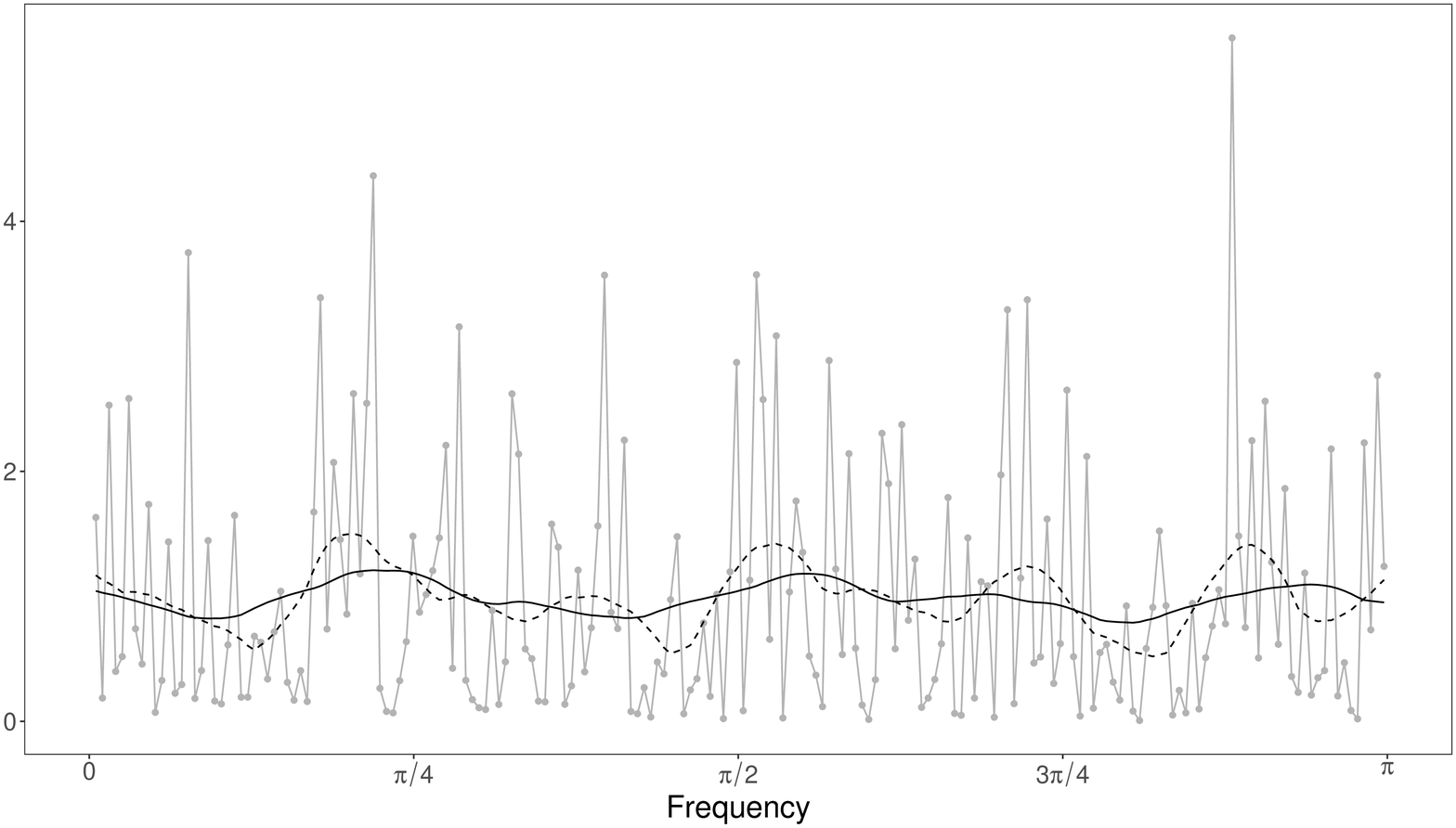}
\caption{Plot of the normalised periodogram ordinates $I(\omega_j)/f_{\widehat\theta}(\omega_j)$ (grey lines) and of the kernel estimate $\widehat q(\omega_j, \widehat\theta)$ with bandwidths $h = 0.05$ (black dashed curve) and $h = 0.10$ (black solid curve).}
\label{fig:goodnessOfFit1}
\end{subfigure} 

\vspace{2ex}
\begin{subfigure}[t]{\textwidth}
\centering
\includegraphics[width=.9\textwidth]{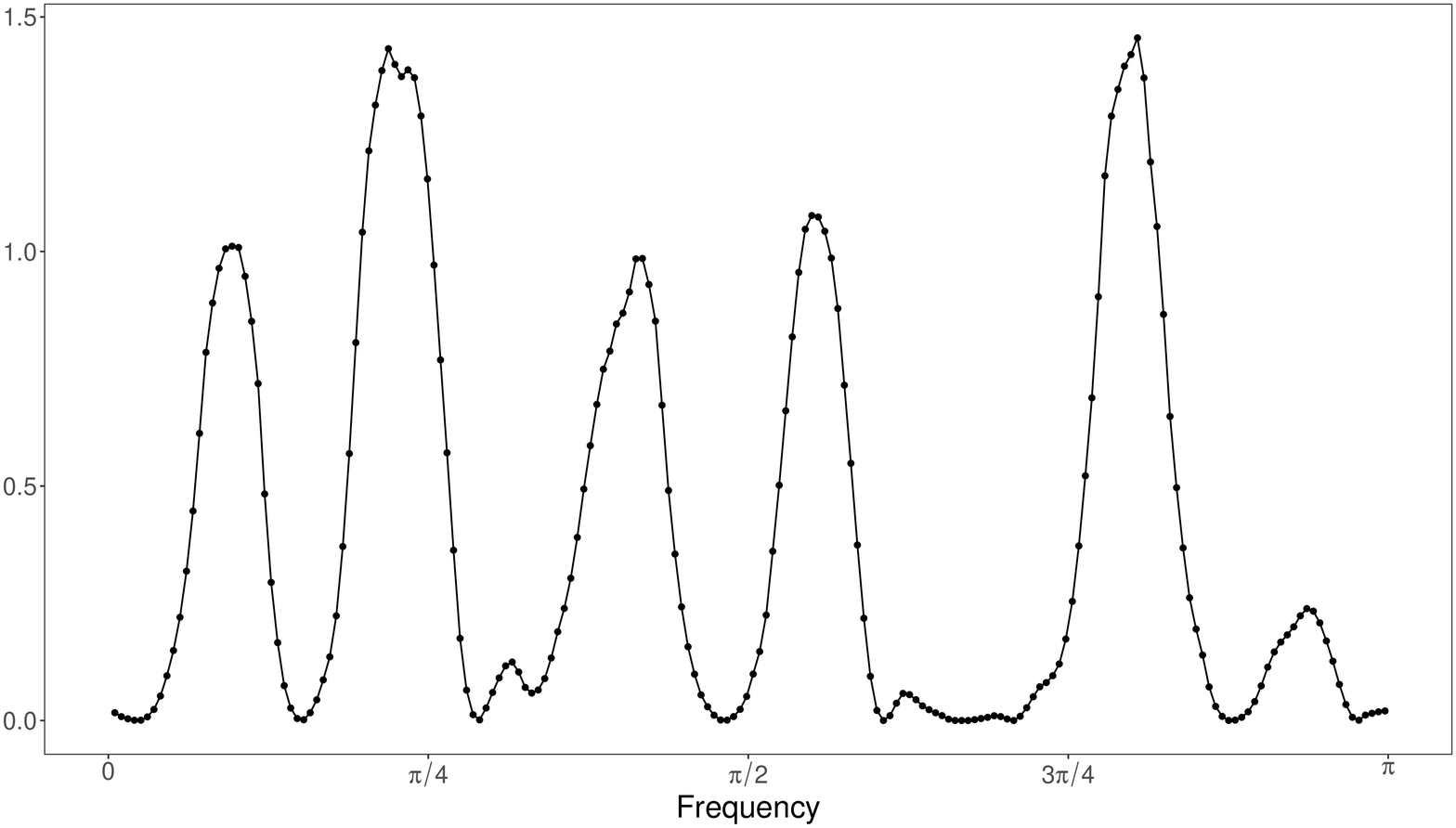}
\caption{Plot of the statistic $Q^2(\omega_j, \widehat\theta)$ with bandwidth $h=0.10$. Note that the value of the 95\% quantile of the $\chi_1^2$ distribution is 3.84, well above the drawing box of the plot.}
\label{fig:goodnessOfFit2}
\end{subfigure}
\caption{Goodness-of-fit diagnostic plots for the measles dataset.}
\end{figure} 

\section{Conclusion}
\label{sec:conclusion}
In this article, we establish a strong mixing condition with polynomial decay rate for stationary Hawkes processes, then propose a Whittle estimation procedure from their count data.
To our knowledge, this is the first work investigating strong mixing conditions for the estimation of Hawkes processes.
This approach has appealing features: \textit{(i)} it has good asymptotic properties, similar to maximum likelihood estimation; \textit{(ii)} it is easy to implement and flexible, since the only user-specified input is the Fourier transform $\widetilde h$ of the reproduction kernel $h^\ast$; \textit{(iii)} it is computationally efficient, with a complexity of $\mathcal O(n\,\mathrm{log}~n)$, $n$ the number of bins, from calculating the periodogram with a fast Fourier transform, compared to $\mathcal O(p^2)$, $p$ the number of events, for the maximum likelihood method (except when the kernel is exponential, in which case the complexity is reduced to $\mathcal O(p)$ with minimal efforts \citep{Ozaki1979}, making it more efficient than our approach); \textit{(iv)} it is particularly well-adapted to applications where the bin size cannot be chosen arbitrarily, \textit{i.e.} the events are only counted in bins of fixed size.


\begin{appendix}
\section{Proof of Theorem \ref{THM:ALPHAMIXING} and Corollary \ref{COR:NONSTATIONARY}}
\label{sec:proof}
By definition, for a given Hawkes process $N$, we have
\begin{equation*}
\alpha_N(r) \coloneqq \underset{t \in \mathbb R}{\mathrm{sup}} 
~\alpha\big(\mathcal E_{-\infty}^t, \mathcal E_{t+r}^\infty\big)
= \underset{t \in \mathbb{R}}{\mathrm{sup}} 
~\underset{\begin{subarray}{c}
  \mathcal{A} \in \mathcal{E}_{-\infty}^t \\
  \mathcal{B} \in \mathcal{E}_{t+r}^\infty
\end{subarray}}
{\mathrm{sup}}  ~\big| \mathrm{Cov}\big(\mathbbm{1}_\mathcal{A}(N), \mathbbm{1}_\mathcal{B}(N)\big) \big|,
\end{equation*}
where $\mathbbm{1}_\mathcal{A}(N)$ is the indicator function of the cylinder set $\mathcal A$, \textit{i.e.} for an elementary cylinder set $\mathcal A_{B, m} = \{N \in \mathfrak N : N(B) = m\}$, $\mathbbm 1_{\mathcal A_{B, m}}(N) = 1$ if $N(B) = m$ and $0$ otherwise.

We recall that a point process $N$ is said to be positively associated if, for all families of pairwise disjoint Borel sets $(A_i)_{1 \le i \le k}$ and $(B_j)_{1 \le j \le l}$, and for all coordinate-wise increasing functions $F : \mathbb{N}^k \rightarrow \mathbb{R}$ and $G : \mathbb{N}^l \rightarrow \mathbb{R}$, it satisfies
\begin{equation*}
\mathrm{Cov} \Big(F\big(N(A_1), \ldots, N(A_k)\big), G\big(N(B_1), \ldots, N(B_l)\big)\Big) \ge 0.
\end{equation*}
We start by stating a useful property (see \citealp[Section 2.1, key property (e)]{Gao2018}), which follows from Hawkes processes being infinitely divisible processes:
\begin{proposition}
\label{prop:pa}
The Hawkes process is positively associated.
\end{proposition}

Using this proposition and \citeauthor{Poinas2017}'s work on associated point processes \citep{Poinas2017}, the following lemma controls the covariance of the indicator functions by the covariance of the count measure of the process, then rescale the problem to a single branching process, thanks to the independence between clusters of a Hawkes process. 
\begin{lemma}
\label{lem:covcond}
Let $s,t,u \in \mathbb{R}$ and $r > 0$ such that $s < t < t+r < u$, and let $\mathcal{A} \in \mathcal{E}_s^t, \mathcal{B} \in \mathcal{E}_{t+r}^{u}$. Then,
\begin{equation*}
\big| \mathrm{Cov} \big(\mathbbm{1}_\mathcal{A}(N), \mathbbm{1}_\mathcal{B}(N)\big) \big| \le \int \left| \mathrm{Cov}\Big(N\big((s,t] \big| y\big), N\big((t+r, u] \big| y\big)\Big) \right| M_c(\dd y)
\end{equation*}
where $N(\cdot \vert y)$ denotes the branching process consisting of an immigrant at time $y$ and all its descendants, and $M_c(\cdot)$ refers to the first-order moment of the centre process $N_c$.
\end{lemma}
\begin{proof}
Using Proposition \ref{prop:pa} and \citep[Theorem~2.5]{Poinas2017}, we have
\begin{equation*}
\big| \mathrm{Cov} \big(\mathbbm{1}_\mathcal{A}(N), \mathbbm{1}_\mathcal{B}(N)\big) \big| \le \Big| \mathrm{Cov}\Big(N\big((s, t]\big), N\big((t+r, u]\big) \Big) \Big|.
\end{equation*}
Then, conditioning on the cluster centre process $N_c$ (see for example \citealp[Exercise 6.3.4]{Daley2003}):
\begin{align*}
\mathrm{Cov}\Big(N\big((s, t]\big), N\big((t+r, u]\big) \Big) 
&= \int \mathrm{Cov}\Big(N\big((s, t] \big| y\big), N\big((t+r, u] \big| y\big)\Big) M_c(\dd y)\\
&+ \int \mathbb{E} \Big[N\big((s, t] \big| x\big)\Big] \mathbb{E} \Big[N\big((t+r, u] \big| y\big)\Big] C_c(\dd x \times \dd y),
\end{align*}
where $M_c(\cdot)$ and $C_c(\cdot)$ refer to the first-order moment measure and the covariance measure of the centre process $N_c$ respectively.
Since the centre process is Poisson, $C_c \equiv 0$ and the second term is zero.
\end{proof}

We are now interested in deriving an upper bound for the covariance of counts of a typical single branching process $N(\cdot | y)$.
Without loss of generality, we consider a cluster whose immigrant is located at time $y = 0$.
Let $Z_k$ denote the number of points of generation $k$, and by $Z_k^{(s,t]}$ those that are located in the interval $(s, t]$ ($s, t \in \mathbb R$).
Note that, for generation $0$, there is a single immigrant located at time $0$.
By definition, we have
\begin{equation*}
N\big((s, t] \big| 0 \big)= \sum_{k=0}^{+\infty} Z_k^{(s,t]}.
\end{equation*}
Then, the covariance between two intervals for a branching process is
\begin{equation*}
\mathrm{Cov}\Big(N\big((s, t] \big| 0\big), N\big((t+r, u] \big| 0\big) \Big) = \sum_{k=0}^{+\infty} \sum_{l=0}^{+\infty} \mathrm{Cov}\left(Z_k^{(s,t]}, Z_l^{(t+r, u]}\right).
\end{equation*}

Before continuing further, we will need a few results on the Galton--Watson process $(Z_k)_{k \in \mathbb{N}}$:
\begin{lemma}
For $k \ge 0$, the expectation, variance and second-order moment of $Z_k$ are
\begin{align*}
\mathbb{E}[Z_k] &= \mu^k, \\
\mathrm{Var}(Z_k) &= \mu^k \sum_{j=0}^{k-1} \mu^j = \mu^k \frac{1-\mu^{k}}{1-\mu}, \\
\mathbb{E}[Z_k^2] &= \mu^k \sum_{j=0}^{k} \mu^j = \mu^k \frac{1-\mu^{k+1}}{1-\mu}.
\end{align*}
\end{lemma}
\begin{proof}
Call $\phi_k$ the probability-generating function of $Z_k$:
\begin{equation*}
\forall s \in [0,1], \phi_k(s) = \mathbb{E}[s^{Z_k}].
\end{equation*}
It is well-known, for a Galton--Watson process, that $(\phi_k)_{k \in \mathbb{N}}$ verifies
\begin{equation*}
\forall k \in \mathbb{N}, \phi_{k+1} = \phi_k \circ \phi_1
\end{equation*}
where in our case $\phi_1$ is the probability-generating function of a Poisson process with parameter $\mu$.
Differentiating the recurrence relation up to order 2 then evaluating it at $s = 1$ gives the following relations:
\begin{align*}
\phi_{k+1}^\prime(1) &= \phi_1^\prime(1) \phi_k^\prime(1),\\
\phi_{k+1}^{\prime\prime}(1) &= \phi_1^{\prime\prime}(1) \phi_k^\prime(1) + (\phi_1^\prime(1))^2 \phi_k^{\prime\prime}(1),
\end{align*}
where $\phi_k^\prime(1)$ and $\phi_k^{\prime\prime}(1)$ are related to the moments of the process by
\begin{equation*}
\mathbb{E}[Z_k] = \phi_k^\prime(1), \qquad\qquad\qquad \mathrm{Var}(Z_k) = \phi_k^{\prime\prime}(1) + \phi_k^\prime(1) - (\phi_k^\prime(1))^2.
\end{equation*}
Finally plugging in the initial conditions for the Poisson variable $Z_1$, $\phi_1^\prime(1) = \mu$ and $\phi_1^{\prime\prime}(1) = \mu^2$, yields the expected result.
\end{proof}

\begin{lemma}
\label{lem:crossedmoments}
For $0 \le k \le l$, the covariance and second-order product moment of $(Z_k)$ are
\begin{align*}
\mathrm{Cov} ( Z_k, Z_l ) &= \mu^{l} \sum_{j=0}^{k-1} \mu^j = \mu^{l} \frac{1-\mu^{k}}{1-\mu}, \\
\mathbb{E} [Z_k Z_l] &= \mu^{l} \sum_{j=0}^{k} \mu^j = \mu^{l} \frac{1-\mu^{k+1}}{1-\mu}.
\end{align*}
\end{lemma}
\begin{proof}
This is a straightforward recurrence, noting that
\begin{align*}
\mathrm{Cov}(Z_k, Z_{k+h}) &= \mathrm{Cov}\left( Z_k,  \sum_{i=1}^{+\infty} \mathbbm{1}_{\{Z_{k+h-1} \ge i\}} Z_{1, i} \right)\\
&= \mathbb{E}\left[Z_{1,1}\right] \mathrm{Cov}\left(Z_k, \sum_{i=1}^{+\infty} \mathbbm{1}_{\{Z_{k+h-1} \ge i\}}\right)\\
&= \mu \, \mathrm{Cov}(Z_k, Z_{k+h-1}).
\end{align*}
wherein $Z_{1,i}$ denotes the number of offsprings of the point $i$ of generation $k+h-1$, is independent of $Z_{1,j}$ ($i \ne j$), of $Z_{k+h-1}$ and of $Z_k$, and has the same distribution as $Z_1$.
\end{proof}

Let $T_i^k$ denote the time of arrival of the $i$-th point of generation $k$.
It has a parent $T_j^{k-1}$ (when $k > 0$).
Let $\Delta^k_i$ be the associated inter-arrival time, \textit{i.e.} $\Delta^k_i = T_i^k - T_j^{k-1}$.
Then, for each point $i$ of generation $k$, there exists a sequence $(\alpha_{i,k}^{(j)})_{1 \le j \le k}$, with $\alpha_{i,k}^{(k)} = i$, denoting the indices of the ancestors of $T_i^k$, such that 
\begin{equation*}
T_i^k = \sum_{j=1}^k \Delta_{\alpha_{i,k}^{(j)}}^j.
\end{equation*}
As a consequence, we get the following lemma:
\begin{lemma}
\label{lem:vonbahr}
For $k \in \mathbb{N}$ and $1 \le i, j \le Z_k$,
\begin{itemize}
\item[(i)] $T_i^k$ and $T_j^k$ are identically distributed, with distribution function equal to the $k$-multiple convolution of $h^\ast$ with itself,
\item[(ii)] For $\delta > 0$, there is an upper bound on the $m$-th moment of $T_1^k$:
\begin{equation*}
\mathbb{E}\big[(T_1^k)^{1+\delta}\big] \le k^{1+\delta} ~\mathbb{E}\big[(\Delta_1^1)^{1+\delta}\big] = k^{1+\delta} \, \nu_{1+\delta}
\end{equation*}
where $\nu_{1+\delta} \coloneqq \int_\mathbb{R} t^{1+\delta} h^\ast(t)\dd t$.
\end{itemize}
\end{lemma}
\begin{proof}
Statement (i) follows from the variables $\Delta_i^k$ being independent of $\Delta_j^l$ for $(i, k) \ne (j, l)$, and identically distributed with density function $h^\ast$.
For statement (ii), the upper bound of the $(1+\delta)$-th order moment of $T_1^k$ can be obtained using the following Hölder's inequality:
	\begin{equation*}
		T_1^k 
		= \sum_{j=1}^k \Delta_{\alpha_{1,k}^{(j)}}^j
		\le \left( \sum_{j=1}^k 1 \right)^\frac{\delta}{1+\delta} \left( \sum_{j=1}^k \left(\Delta_{\alpha_{1,k}^{(j)}}^j\right)^{1+\delta} \right)^\frac{1}{1+\delta}.
	\end{equation*}
\end{proof}

Additionally, since for any point of the branching process offsprings are generated by a Poisson process, the arrival times, say $\Delta_i^k$, are independent from the number of offsprings generated at 
the current or past generations.
Conversely, since the reproduction mean $\mu$ does not depend on the time, the number of offsprings generated at any generation, say $Z_l$, are independent from the past arrival times.
Consequently, we have the following lemma:
\begin{lemma}
\label{lem:indep}
For $k,l \in \mathbb{N}$ and $1 \le i \le Z_k$,
$T_i^k$ and $Z_l$ are independent.
\end{lemma}

\begin{remark} 
This lemma separates the genealogy of the Galton--Watson process $(Z_k)$ from the arrival times $(T_i^k)$ of the branching process, analogously to how the Poisson process is a binomial process with Poisson-distributed number of points.
Then, a cluster in a Hawkes process is equivalent to a Galton--Watson process $(Z_k)$, upon which the ancestors $(\alpha_{i,k}^{(k-1)})$ are drawn equiprobably from the $Z_{k-1}$ possible ancestors and the $(\Delta_i^k)$ independently with distribution function $h^\ast$.
Intuitively, since each point $j$ of generation $k-1$ generates offsprings according to the same intensity measure, then each point of generation $k$ has ancestor $j$ with equiprobability.
\end{remark}

Finally, we will need the following identity for the covariance of the product of independent random variables.
\begin{lemma}
\label{lem:cov}
Let $(X_i^k)_{i, k \in \mathbb{N}}$ and $(Y_j^l)_{j, l \in \mathbb{N}}$ be two collections of random variables such that, for all $i, j, k, l \in \mathbb{N}$, the variables $X_i^k$ and $Y_j^l$ are independent.
Then
\begin{equation*}
\mathrm{Cov}(X_i^k Y_i^k, X_j^l Y_j^l) = \mathbb{E}[X_i^k X_j^l]\,\mathrm{Cov}(Y_i^k, Y_j^l) + \mathbb{E}[Y_i^k]\, \mathbb{E}[Y_j^l]\, \mathrm{Cov}(X_i^k, X_j^l).
\end{equation*}
\end{lemma}
\begin{proof}
Writing the expression of the covariance then adding and substracting the term $\mathbb{E}[X_i^k X_j^l]\,\mathbb{E}[Y_i^k]\,\mathbb{E}[Y_j^l]$ yields the relation.
\end{proof}

We can now derive an upper bound for $\mathrm{Cov}\left(Z_k^{(s,t]}, Z_l^{(t+r, u]}\right)$:
\begin{lemma}
\label{lem:onecluster}
Let $s,t,u \in \mathbb{R}$ and $r > 0$ such that $s < t < t+r < u$.
Suppose that there exists $\delta > 0$ such that $\nu_{1+\delta} < \infty$.
Then
\begin{equation*}
\left| \mathrm{Cov}\left(Z_k^{(s,t]}, Z_l^{(t+r, u]}\right) \right| \le 2 \, \frac{l^{1+\delta} \, \nu_{1+\delta}}{(t+r)^{1+\delta}} \, \mu^{k\vee l} \, \frac{1-\mu^{k\wedge l+1}}{1-\mu},
\end{equation*}
where $k\vee l = \mathrm{max}\,(k, l)$ and $k\wedge l = \mathrm{min}\,(k, l)$.
\end{lemma}
\begin{proof}
We have
\begin{align*}
\mathrm{Cov}\left(Z_k^{(s,t]}, Z_l^{(t+r, u]}\right) 
&= \mathrm{Cov}\left(\sum_{i=1}^{Z_k} \mathbbm{1}_{\{T_i^k \in (s, t]\}}, \sum_{j=1}^{Z_l} \mathbbm{1}_{\{T_j^l \in (t+r, u]\}}\right)\\
&= \sum_{i=1}^{+\infty} \sum_{j=1}^{+\infty} \mathrm{Cov}\left( \mathbbm{1}_{\{Z_k \ge i\}} \mathbbm{1}_{\{T_i^k \in (s, t]\}},  \mathbbm{1}_{\{Z_l \ge j\}} \mathbbm{1}_{\{T_j^l \in (t+r, u]\}}\right).
\end{align*}
Then, by Lemmas \ref{lem:indep} and \ref{lem:cov},
\begin{align*}
&\mathrm{Cov}\left( \mathbbm{1}_{\{Z_k \ge i\}} \mathbbm{1}_{\{T_i^k \in (s, t]\}},  \mathbbm{1}_{\{Z_l \ge j\}} \mathbbm{1}_{\{T_j^l \in (t+r, u]\}}\right)\\
&\qquad\qquad = \mathbb{E} \left[ \mathbbm{1}_{\{Z_k \ge i\}} \mathbbm{1}_{\{Z_l \ge j\}} \right] \mathrm{Cov}\left( \mathbbm{1}_{\{T_i^k \in (s, t]\}}, \mathbbm{1}_{\{T_j^l \in (t+r, u]\}}\right) \\
&\qquad\qquad + \mathbb{E}\left[ \mathbbm{1}_{\{T_i^k \in (s, t]\}} \right] \mathbb{E}\left[ \mathbbm{1}_{\{T_j^l \in (t+r, u]\}} \right] \mathrm{Cov}\left(\mathbbm{1}_{\{Z_k \ge i\}}, \mathbbm{1}_{\{Z_l \ge j\}}\right).
\end{align*}

\noindent For the first term,
\begin{align*}
\mathrm{Cov}\left( \mathbbm{1}_{\{T_i^k \in (s, t]\}}, \mathbbm{1}_{\{T_j^l \in (t+r, u]\}} \right) 
&= \mathbb{E} \left[ \mathbbm{1}_{\{T_i^k \in (s, t]\}} \mathbbm{1}_{\{T_j^l \in (t+r, u]\}} \right] - \mathbb{E} \left[ \mathbbm{1}_{\{T_i^k \in (s, t]\}} \right] \mathbb{E} \left[ \mathbbm{1}_{\{T_j^l \in (t+r, u]\}} \right]\\
& \le \mathbb{E} \left[ \mathbbm{1}_{\{T_j^l \in (t+r, u]\}} \right] \\
& \le \mathbb{P} \left( T_j^l \ge t+r \right)\\
& \le \frac{\mathbb{E} \left[ (T_1^l)^{1+\delta} \right]}{(t+r)^{1+\delta}}\\
& \le \frac{l^{1+\delta} \, \nu_{1+\delta}}{(t+r)^{1+\delta}},
\end{align*}
using Markov's inequality for the second to last inequality, and Lemma \ref{lem:vonbahr} for the last one.
Similarly, 
\begin{align*}
\mathrm{Cov}\left( \mathbbm{1}_{\{T_i^k \in (s, t]\}}, \mathbbm{1}_{\{T_j^l \in (t+r, u]\}} \right) 
&= \mathbb{E} \left[ \mathbbm{1}_{\{T_i^k \in (s, t]\}} \mathbbm{1}_{\{T_j^l \in (t+r, u]\}} \right] - \mathbb{E} \left[ \mathbbm{1}_{\{T_i^k \in (s, t]\}} \right] \mathbb{E} \left[ \mathbbm{1}_{\{T_j^l \in (t+r, u]\}} \right]\\
& \ge - \mathbb{E} \left[ \mathbbm{1}_{\{T_j^l \in (t+r, u]\}} \right] \\
& \ge - \frac{l^{1+\delta} \, \nu_{1+\delta}}{(t+r)^{1+\delta}},
\end{align*}

\noindent The second term is straightforward,
\begin{align*}
\left| \mathbb{E}\left[ \mathbbm{1}_{\{T_i^k \in (s, t]\}} \right] \mathbb{E}\left[ \mathbbm{1}_{\{T_j^l \in (t+r, u]\}} \right] \right|
&\le \mathbb{E}\left[ \mathbbm{1}_{\{T_j^l \in (t+r, u]\}} \right] \\
&\le \frac{l^{1+\delta} \, \nu_{1+\delta}}{(t+r)^{1+\delta}}.
\end{align*}

\noindent Then:
\begin{align*}
&\left| \sum_{i=1}^{+\infty} \sum_{j=1}^{+\infty} \mathrm{Cov}\left( \mathbbm{1}_{\{Z_k \ge i\}} \mathbbm{1}_{\{T_i^k \in (s, t]\}},  \mathbbm{1}_{\{Z_l \ge j\}} \mathbbm{1}_{\{T_j^l \in (t+r, u]\}}\right) \right|\\
& \qquad \le \frac{l^{1+\delta} \, \nu_{1+\delta}}{(t+r)^{1+\delta}} \left| \sum_{i=1}^{+\infty} \sum_{j=1}^{+\infty} \mathbb{E} \left[ \mathbbm{1}_{\{Z_k \ge i\}} \mathbbm{1}_{\{Z_l \ge j\}} \right] + \sum_{i=1}^{+\infty} \sum_{j=1}^{+\infty} \mathrm{Cov}\left(\mathbbm{1}_{\{Z_k \ge i\}}, \mathbbm{1}_{\{Z_l \ge j\}}\right) \right|\\
& \qquad = \frac{l^{1+\delta} \, \nu_{1+\delta}}{(t+r)^{1+\delta}} \ \big| \mathbb{E} \left[ Z_k Z_l \right] + \mathrm{Cov} \left( Z_k, Z_l \big) \right|\\
& \qquad \le 2 \, \frac{l^{1+\delta} \, \nu_{1+\delta}}{(t+r)^{1+\delta}} \: \mu^{k\vee l} \: \frac{1-\mu^{k\wedge l+1}}{1-\mu},
\end{align*}
using Lemma \ref{lem:crossedmoments} for the last inequality.
\end{proof}

Straightforwardly, since $\sum \mu^k$ and $\sum l^{1+\delta} \mu^l$ are summable for $\delta > 0$, we get the following lemma:
\begin{lemma}
\label{lem:bigO}
Let $s,t,u \in \mathbb{R}$ and $r > 0$ such that $s < t < t+r < u$.
Suppose that there exists $\delta > 0$ such that $\, \nu_{1+\delta} < \infty$. Then,
\begin{equation*}
\left| \mathrm{Cov}\Big(N\big((s, t] \big| 0\big), N\big((t+r, u] \big| 0\big) \Big) \right| = \mathcal{O} \left( \frac{1}{(t+r)^{1+\delta}} \right).
\end{equation*}
\end{lemma}

All that is left to prove Theorem \ref{THM:ALPHAMIXING} and Corollary \ref{COR:NONSTATIONARY} is to integrate the upper bound with respect to the first-moment measure of the  centre process.
Using the notations of Lemmas \ref{lem:covcond} and \ref{lem:bigO}, and with $M_c(\cdot) = \eta(\cdot) \ell(\cdot)$ where $\ell(\cdot)$ is the Lebesgue measure,
\begin{align*}
\big| \mathrm{Cov} \big(\mathbbm{1}_\mathcal{A}(N), \mathbbm{1}_\mathcal{B}(N)\big) \big| 
&\le \int_\mathbb{R} \Big| \mathrm{Cov}\Big(N\big((s, t] \big| y\big), N\big((t+r, u] \big| y\big)\Big) \Big| \; M_c(\dd y)\\
&= \int_{-\infty}^t \Big| \mathrm{Cov}\Big(N\big((s, t] \big| y\big), N\big((t+r, u] \big| y\big)\Big) \Big| \; M_c(\dd y)\\
&= \mathcal{O} \left( \int_{-\infty}^t \frac{1}{(t+r-y)^{1+\delta}} \eta(y) \dd y \right)\\
&= \mathcal{O} \left( r^{-\delta} \right),
\end{align*}
where the last inequality follows from the boundedness of $\eta(\cdot)$.
This upper bound is valid for any $s, u \in \mathbb{R}$, therefore holds for $\mathcal{A} \in \mathcal{E}_{-\infty}^t, \mathcal{B} \in \mathcal{E}_{t+r}^\infty$.

We now turn to an upper bound for the strong mixing coefficient when the reproduction kernel $h^\ast$ admits a finite exponential moment, that is, there exists $a_0 > 0$ such that 
\begin{equation*}
    M(a_0) \coloneqq \int_{\mathbb R} e^{a_0 |t|} h^\ast(t) \dd t < \infty.
\end{equation*}
Choose $a \in (0, a_0]$ such that $1 < M(a) < \mu^{-1}$.
Then, by substituting in Lemma \ref{lem:onecluster} the Markov inequality with:
\begin{align*}
    \mathbb{P} \big( T_j^l \ge t+r \big) 
    &\le \mathbb E\big[\exp\big(a |T_1^l|\big)\big] e^{-a(t+r)}\\
    &\le \mathbb E\big[\exp\big(a |\Delta_1^1|\big)\big]^l e^{-a(t+r)}\\
    &= M(a)^l e^{-a(t+r)},
\end{align*}
the term $\sum M(a)^l \mu^l$ is again summable, and Lemma \ref{lem:bigO} turns into:
\begin{equation*}
\left| \mathrm{Cov}\Big(N\big((s, t] \big| 0\big), N\big((t+r, u] \big| 0\big) \Big) \right| = \mathcal{O} \left( e^{-a(t+r)} \right).
\end{equation*}
Finally, by integrating with respect to the first-moment measure of the centre process, we get the desired result:
\begin{equation*}
    \big| \mathrm{Cov} \big(\mathbbm{1}_\mathcal{A}(N), \mathbbm{1}_\mathcal{B}(N)\big) \big| = \mathcal{O} \left( e^{-ar} \right).
\end{equation*}
\begin{flushright}
$\blacksquare$
\end{flushright}

\section{Figures of Section \ref{SEC:SIMULATIONS}}
\label{sec:figures}

\begin{figure}[hp!]
\centering
\begin{subfigure}[t]{\textwidth}
\centering
\includegraphics[width=\textwidth]{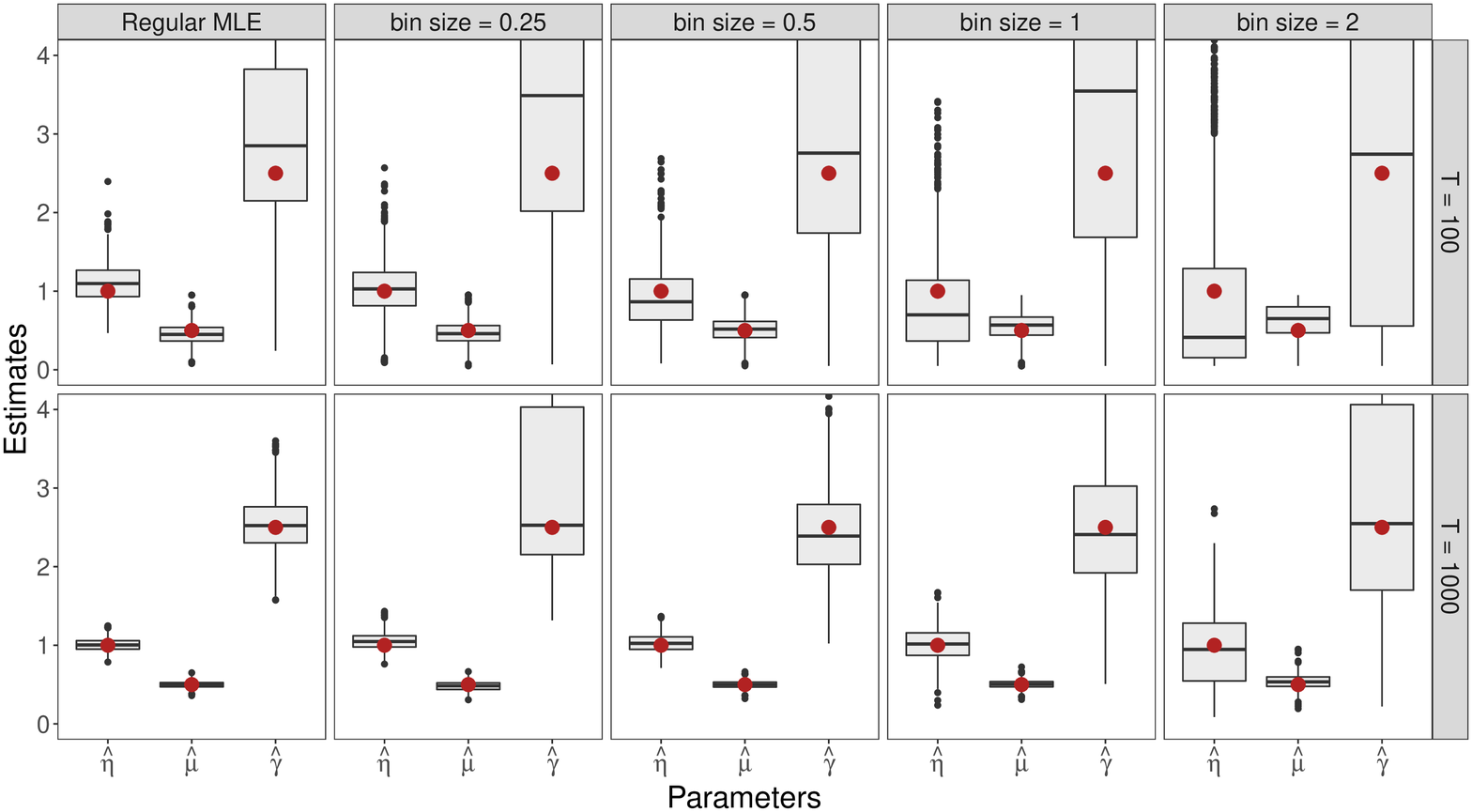}
\caption{Estimates of parameters $\eta$, $\mu$ and $\gamma$ for 1,000 simulations on the interval $[0, T]$.
True values (larger dots) are: $\eta = 1$, $\mu = 0.5$, $\gamma = 2.5$.
The left column refers to the maximum likelihood estimates. The other columns refer to the Whittle estimates according to different bin sizes.}
\label{fig:powerlaw25_estimates}
\end{subfigure} 

\vspace{2ex}
\begin{subfigure}[t]{\textwidth}
\centering
\includegraphics[width=\textwidth]{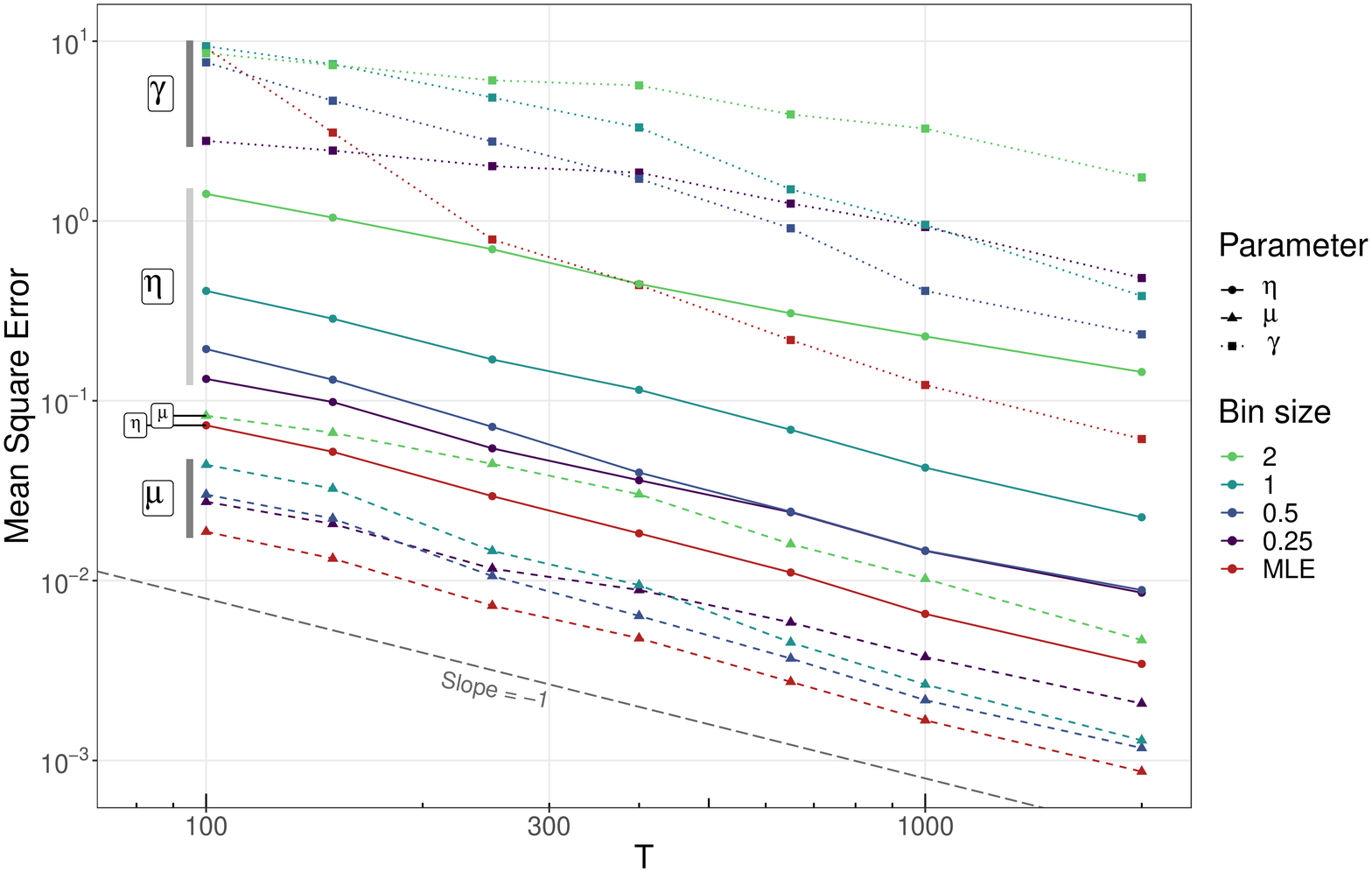}
\caption{Mean square error of the estimates of parameters $\eta$, $\mu$ and $\gamma$ for 1,000 simulations on the interval $[0,T]$, in log-log scale.
The dashed grey line represents the ideal slope of $-1$, \textit{i.e.} a rate of convergence of $\mathcal O(n^{-1})$.}
\label{fig:powerlaw25_convergence}
\end{subfigure}
\caption{Performance of the Whittle estimates for the stationary Hawkes process with immigration intensity $\eta = 1$, reproduction mean $\mu = 0.5$, and reproduction kernel $h^\ast(t) = \gamma a^\gamma (a+t)^{-\gamma - 1}$, where $\gamma = 2.5$ and $a = 1.5$.}
\end{figure} 

\begin{figure}[hp!]
\centering
\begin{subfigure}[t]{\textwidth}
\centering
\includegraphics[width=\textwidth]{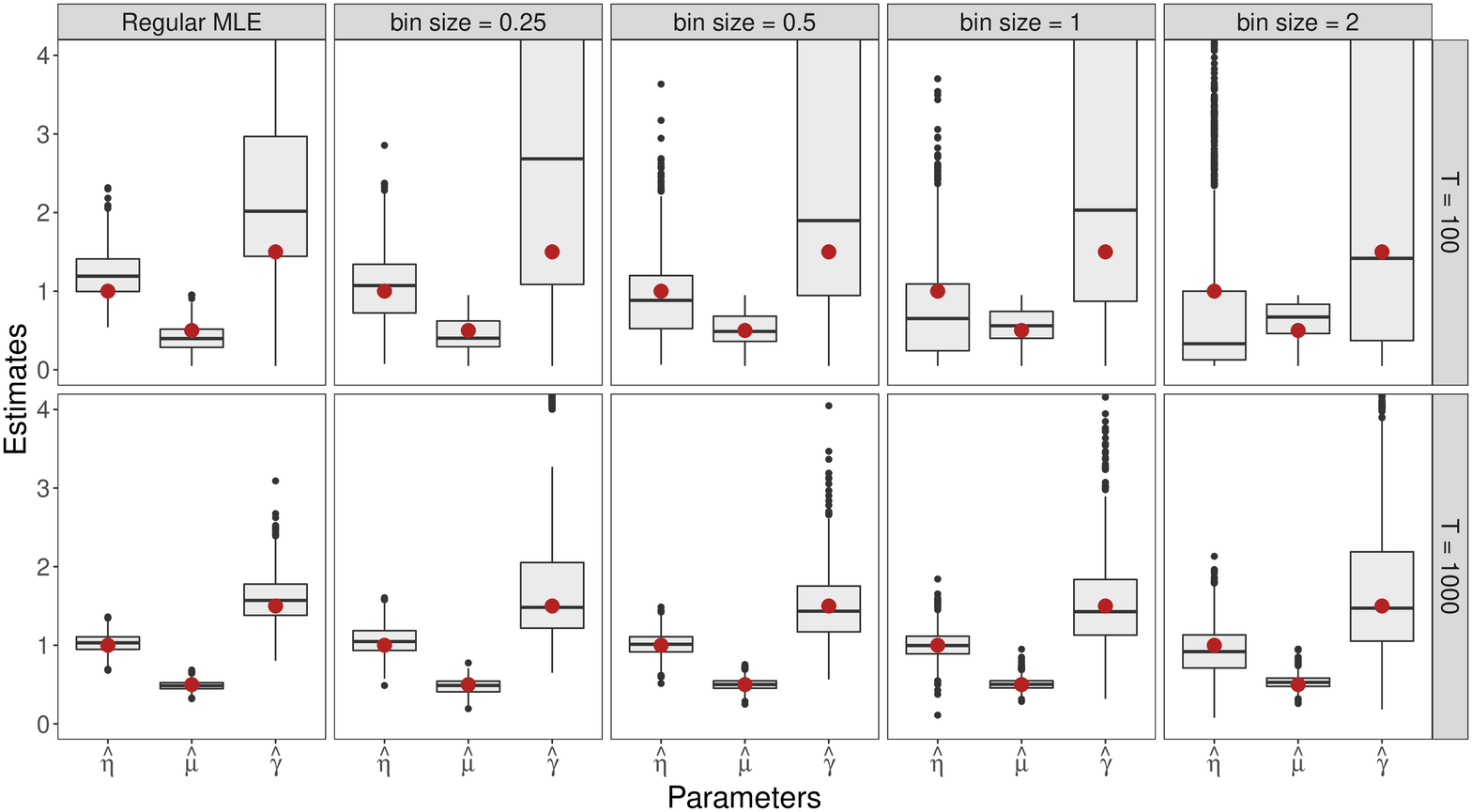}
\caption{Estimates of parameters $\eta$, $\mu$ and $\gamma$ for 1,000 simulations on the interval $[0, T]$.
True values (larger dots) are: $\eta = 1$, $\mu = 0.5$, $\gamma = 1.5$.
The left column refers to the maximum likelihood estimates. The other columns refer to the Whittle estimates according to different bin sizes.}
\label{fig:powerlaw15_estimates}
\end{subfigure} 

\vspace{2ex}
\begin{subfigure}[t]{\textwidth}
\centering
\includegraphics[width=\textwidth]{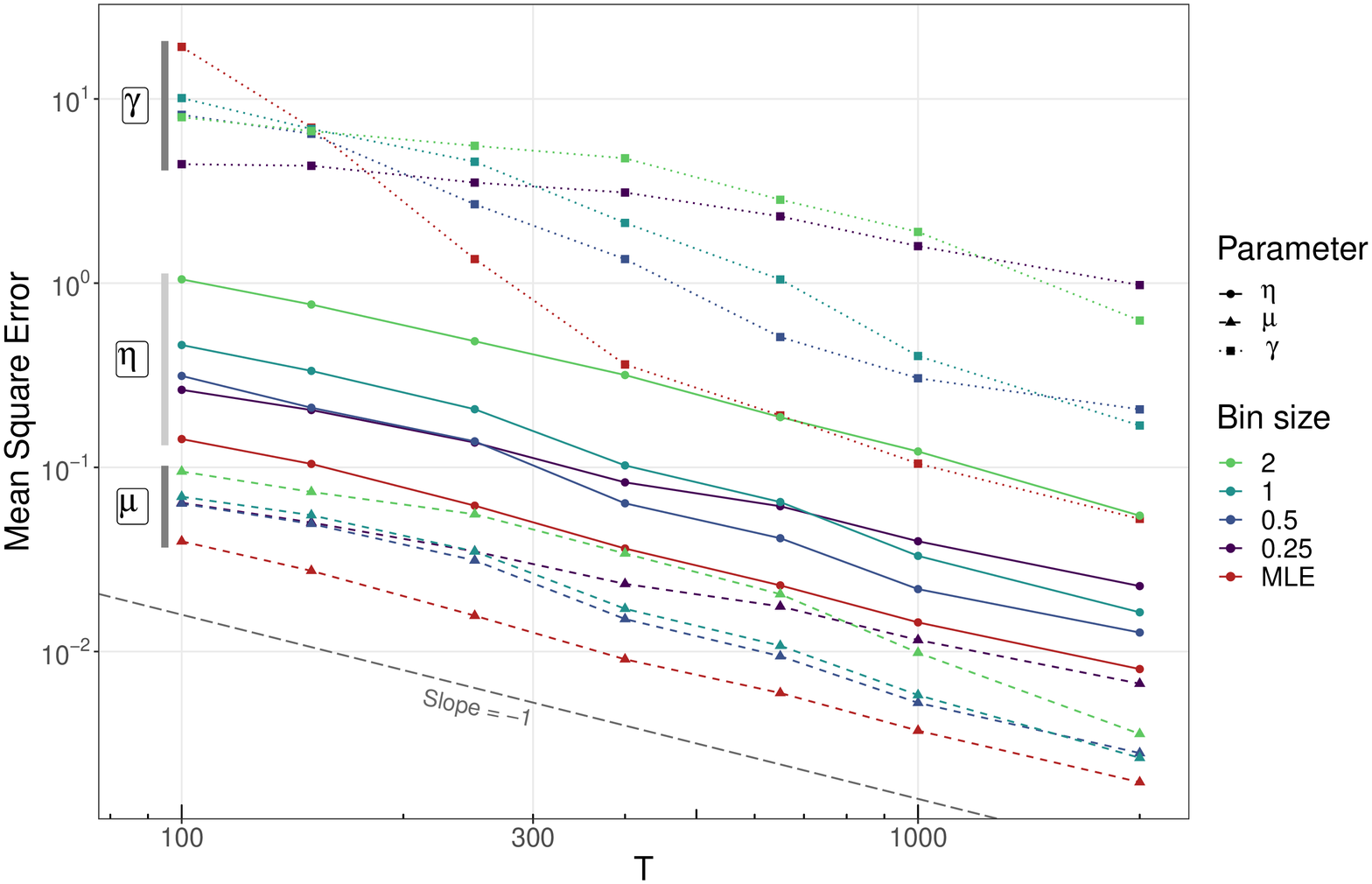}
\caption{Mean square error of the estimates of parameters $\eta$, $\mu$ and $\gamma$ for 1,000 simulations on the interval $[0,T]$, in log-log scale.
The dashed grey line represents the ideal slope of $-1$, \textit{i.e.} a rate of convergence of $\mathcal O(n^{-1})$.}
\label{fig:powerlaw15_convergence}
\end{subfigure}
\caption{Performance of the Whittle estimates for the stationary Hawkes process with immigration intensity $\eta = 1$, reproduction mean $\mu = 0.5$, and reproduction kernel $h^\ast(t) = \gamma a^\gamma (a+t)^{-\gamma - 1}$, where $\gamma = 1.5$ and $a = 1.5$.}
\end{figure} 

\begin{figure}[hp!]
\centering
\begin{subfigure}[t]{\textwidth}
\centering
\includegraphics[width=\textwidth]{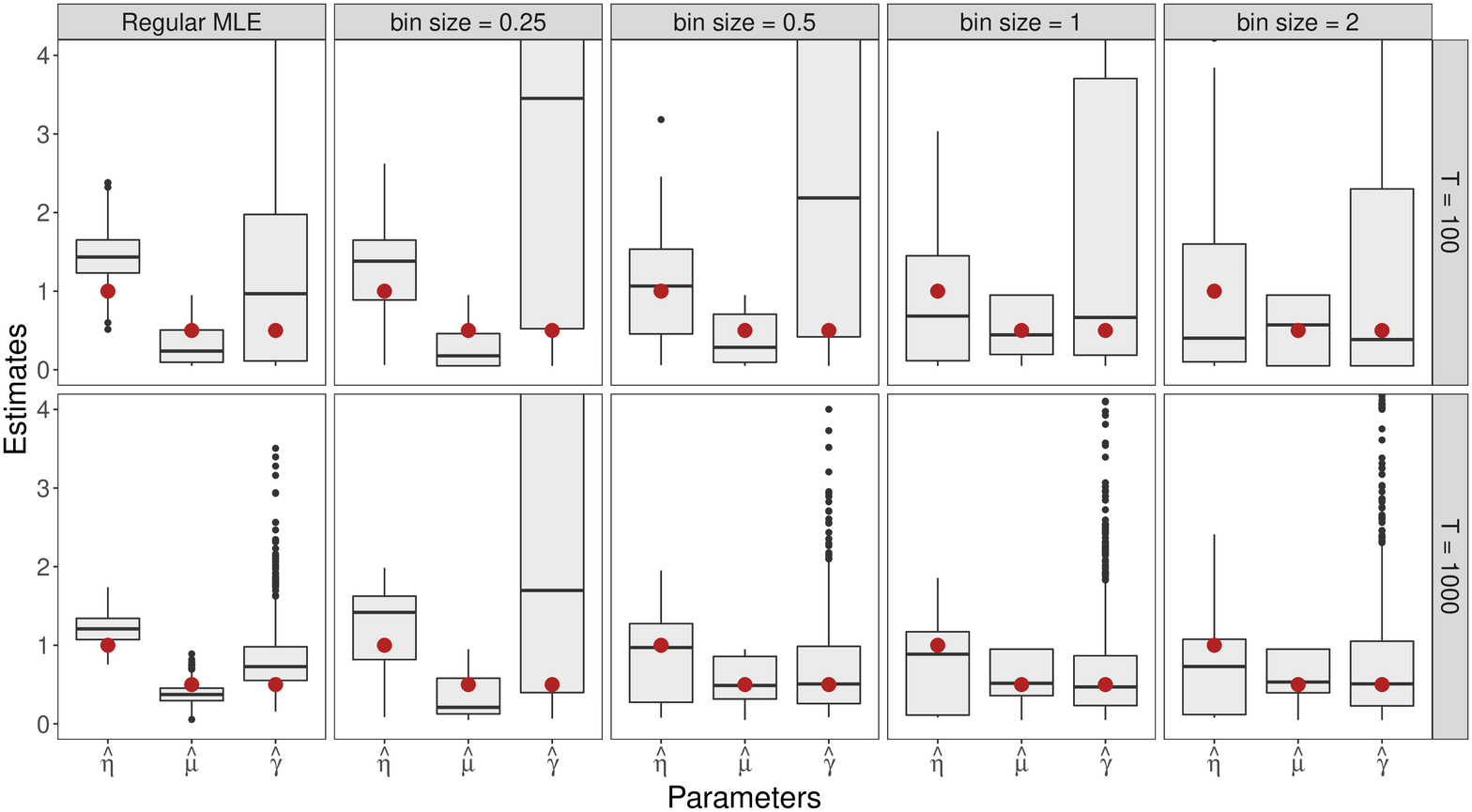}
\caption{Estimates of parameters $\eta$, $\mu$ and $\gamma$ for 1,000 simulations on the interval $[0, T]$.
True values (larger dots) are: $\eta = 1$, $\mu = 0.5$, $\gamma = 0.5$.
The left column refers to the maximum likelihood estimates. The other columns refer to the Whittle estimates according to different bin sizes.}
\label{fig:powerlaw05_estimates}
\end{subfigure} 

\vspace{2ex}
\begin{subfigure}[t]{\textwidth}
\centering
\includegraphics[width=\textwidth]{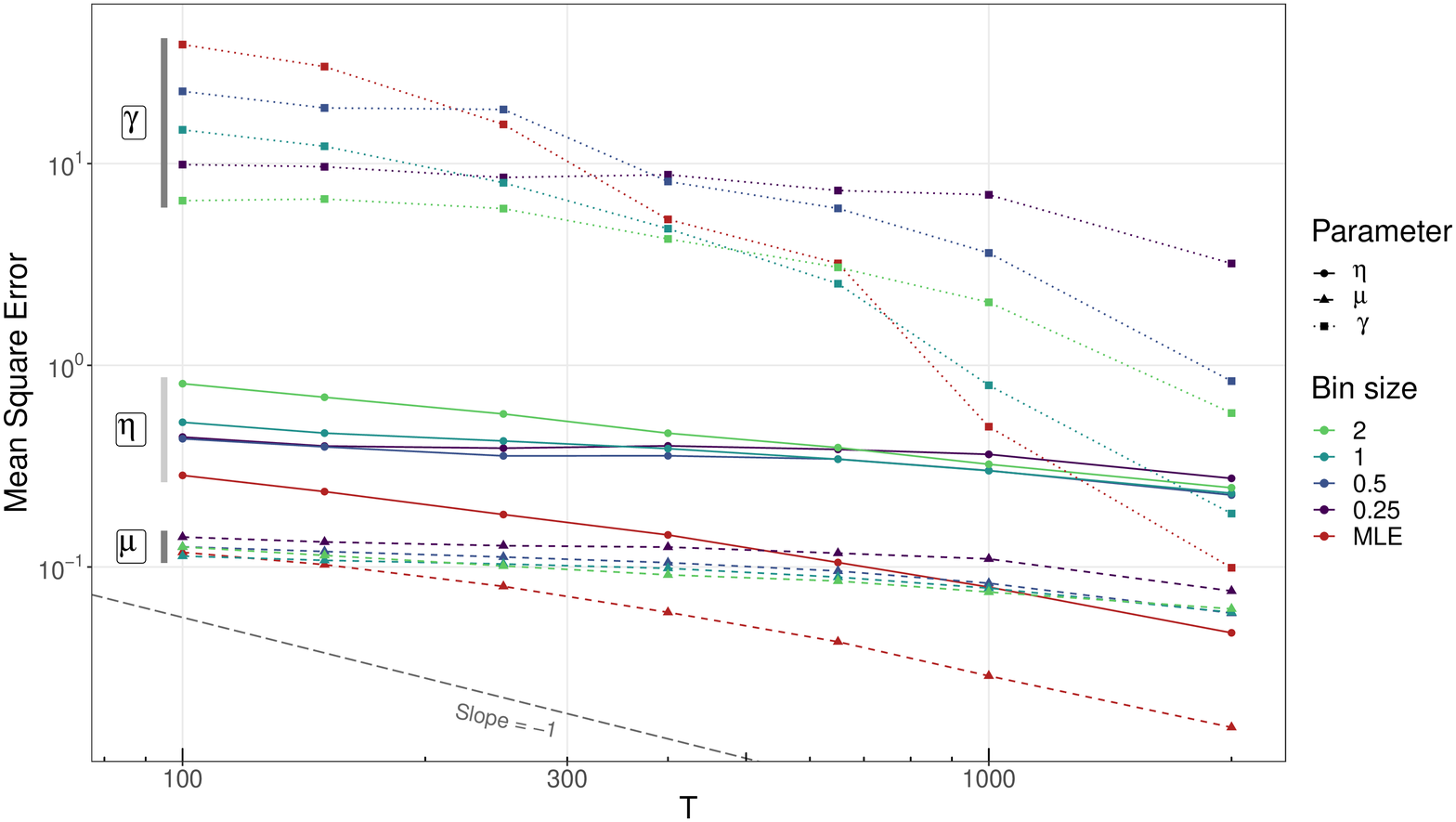}
\caption{Mean square error of the estimates of parameters $\eta$, $\mu$ and $\gamma$ for 1,000 simulations on the interval $[0,T]$, in log-log scale.
The dashed grey line represents the ideal slope of $-1$, \textit{i.e.} a rate of convergence of $\mathcal O(n^{-1})$.}
\label{fig:powerlaw05_convergence}
\end{subfigure}
\caption{Performance of the Whittle estimates for the stationary Hawkes process with immigration intensity $\eta = 1$, reproduction mean $\mu = 0.5$, and reproduction kernel $h^\ast(t) = \gamma a^\gamma (a+t)^{-\gamma - 1}$, where $\gamma = 0.5$ and $a = 1.5$.}
\end{figure} 

\end{appendix}

\section*{Acknowledgements}
The authors would like to thank Fran\c cois Roueff who suggested the use of Whittle's method for the estimation of Hawkes processes from bin-count data and Theorem \ref{THM:ALPHAMIXING}'s extension to exponentially decaying reproduction kernels. 
The authors would also like to thank the three anonymous reviewers who helped, through their remarks and suggestions, to considerably improve this article.
The authors received no specific funding for this work.
During this work, Felix Cheysson was a PhD student of UMR MIA-Paris, Université Paris-Saclay, AgroParisTech, INRAE; Epidemiology and Modeling of bacterial Evasion to Antibacterials Unit (EMEA), Institut Pasteur; and Centre de recherche en Epidémiologie et Santé des Populations (CESP), Université Paris-Saclay, UVSQ, Inserm.

\bibliographystyle{abbrvnatnourl}
\bibliography{references}

\end{document}